\documentclass{amsproc}

\usepackage{amsmath}
\usepackage{amssymb}
\usepackage{amsthm}
\usepackage[mathcal]{eucal}
\usepackage{rotating}
\usepackage{graphicx}
\usepackage[tableposition=below]{caption}
\usepackage[section]{placeins} %to automatically ensure floats do not go into the next section
\usepackage{bm}

\DeclareGraphicsExtensions{.png,.pdf,.eps}
\usepackage[all,2cell,cmtip]{xy}
\usepackage{tikz}
\usepackage{color}
\usepackage{longtable}
\newtheorem{theorem}{Theorem}[section]
\newtheorem*{main}{Main Theorem}
\newtheorem{lemma}[theorem]{Lemma}
\newtheorem{corollary}[theorem]{Corollary}
\newtheorem{proposition}[theorem]{Proposition}

\newtheorem{notation}[theorem]{Notation}
\newtheorem{fact}[theorem]{Fact}

\theoremstyle{definition}

\newtheorem{definition}[theorem]{Definition}

\newtheorem{remark}[theorem]{Remark}

\newtheorem{setup}[theorem]{Setup}

\def\Plus{\texttt{+}}
\def\Minus{\texttt{-}}
\def\A{{\mathbb A}}

\def\C{{\mathbb C}}

\def\P{{\mathbb P}}
\def\Q{{\mathbb Q}}

\def\cC{{\mathcal C}}
\def\cD{{\mathcal D}}
\def\cE{{\mathcal E}}
\def\cF{{\mathcal F}}
\def\cG{{\mathcal{G}}}

\def\cL{{\mathcal L}}
\def\cM{{\mathcal M}}
\def\cN{{\mathcal{N}}}
\def\cO{{\mathcal{O}}}

\def\cS{{\mathcal S}}

\def\cU{{\mathcal U}}

\def\cZ{{\mathcal Z}}
\def\Q{{\mathbb{Q}}}

\def\DA{\operatorname{\hspace{0.01cm} A}}
\def\DB{\operatorname{\hspace{0.01cm} B}}
\def\DC{\operatorname{\hspace{0.01cm} C}}
\def\DD{\operatorname{\hspace{0.01cm} D}}
\def\DE{\operatorname{\hspace{0.01cm} E}}
\def\DF{\operatorname{\hspace{0.01cm} F}}
\def\DG{\operatorname{\hspace{0.01cm} G}}

\def\fg{{\mathfrak g}}

\def\ff{{\mathfrak f}_4}
\def\GL{\operatorname{Gl}}

\def\el{w}
\def\Sc{X}

\def\lra{\longrightarrow}

\def\ra{\rightarrow}
\def\lra{\longrightarrow}

\def\operatorname#1{\mathop{\rm #1}\nolimits}

\def\Aut{\operatorname{Aut}}

\def\Ext{\operatorname{Ext}}

\def\Hom{\operatorname{Hom}}

\def\Pic{\operatorname{Pic}}

\def\id{\operatorname{id}}

\def\NE{{\operatorname{NE}}}

\newcommand{\cNE}[1]{\overline{\NE}(#1)}

\newcommand{\ol}[1]{\overline{#1}}
\newcommand{\pb}{\ar@{}[dr]|(.25){\text{\pigpenfont J}}}
\makeatletter
 
\newcommand*\wthelper[2]{%
        \hbox{\dimen@\accentfontxheight#1%
                \accentfontxheight#11.15\dimen@
                $\m@th#1\widetilde{#2}$%
                \accentfontxheight#1\dimen@
        }%
}

\newcommand*\accentfontxheight[1]{%
        \fontdimen5\ifx#1\displaystyle
                \textfont
        \else\ifx#1\textstyle
                \textfont
        \else\ifx#1\scriptstyle
                \scriptfont
        \else
                \scriptscriptfont
        \fi\fi\fi3
}
\makeatother
\newcommand{\shse}[3]{0 ~\ra ~#1~ \lra ~#2~ \lra ~#3~ \ra~ 0}

\makeindex

\begin{document}
%\pagewiselinenumbers

\title[Deformation of Bott--Samelson varieties]{Deformation of Bott--Samelson varieties and variations of isotropy structures}

\author[Occhetta]{Gianluca Occhetta}
\address{Dipartimento di Matematica, Universit\`a di Trento, via
Sommarive 14 I-38123 Povo di Trento (TN), Italy} \thanks{First author supported by PRIN project ``Geometria delle variet\`a algebriche''. First and second author supported by the Department of Mathematics of the University of Trento.
Second author supported by the Polish National Science Center project 2013/08/A/ST1/00804.  }
\email{gianluca.occhetta@unitn.it, eduardo.solaconde@unitn.it}

\author[Sol\'a Conde]{Luis E. Sol\'a Conde}

\begin{abstract} 
In the framework of the problem of characterizing complete flag manifolds by their contractions, the complete flags of type $\DF_4$ and $\DG_2$ satisfy the property that any possible tower of Bott--Samelson varieties dominating them birationally deforms in a nontrivial moduli. In this paper we illustrate the fact that, at least in some cases, these deformations can be explained in terms of automorphisms of Schubert varieties, providing variations of certain isotropic structures on them. As a corollary, we provide a unified and completely algebraic proof of the characterization of complete flag manifolds in terms of their contractions.
\end{abstract}

\maketitle

\section{Introduction}\label{sec:intro}

Bott--Samelson varieties appear classically in the study of the singular cohomology of complete flag manifolds $G/B$ as desingularizations of Schubert varieties. Introduced in 1958 by Bott and Samelson (\cite{BS58}), they are usually defined as varieties of the form:
$$
(P_1\times^BP_2\times^B\dots\times^BP_r)/B, 
$$%\marginpar{{\color{red} Why is a B down??}}
where $P_1,P_2,\dots ,P_r$ are parabolic subgroups of $G$ containing the Borel subgroup $B$ (see \cite{Dem74}).
Besides their representation-theoretical definition, Bott--Samelson varieties are particularly interesting under the point of view of Mori theory since they may be constructed recursively as towers of $\P^1$-bundles. In this sense, by keeping track of their Mori cones, one may use them to reconstruct rational homogeneous manifolds, More concretely, this idea has been used in \cite{OSWW,OSW} to characterize complete flag manifolds by their $\P^1$-bundle structures; previously a similar result was known only in the case of Picard number two (\cite{kyoto,Wa2}). 

In the general case, one needs to show first that if a smooth complex projective variety $X$ has as many $\P^1$-bundle structures as its Picard number, then  the intersection matrix (see Setup \ref{set:FTvar}) of relative anticanonical bundles and fibers
$$
(-K_i\cdot\Gamma_j)
$$
is equal to the Cartan matrix of a semisimple Lie algebra, determining a complete flag manifold $G/B$. Moreover, from the existence of the $\P^1$-bundle structures one infers a Borel--Weyl--Bott type theorem for line bundles on $X$, which in turn allows   to control, via certain vanishing theorems, the contractions of some Bott--Samelson varieties.

In the second part of the proof one compares $X$ with $G/B$ by means of Bott--Samelson varieties starting from points. One starts by choosing a word $\el$, that is a reduced expression of the longest element in the Weyl group $W$ of $G$, that will determine the towers of Bott--Samelson varieties for $X$ and $G/B$. 

At the $j$-th step of their recursive construction, the corresponding Bott--Samelson variety is determined by (the homothety class of) an element $\theta_j$ of a cohomology group of type $H^1$. The existence of the contraction to $X$ implies that $\theta_j$ must be chosen to be nonzero, whenever the corresponding group $H^1$ is different from zero. In the case in which the Dynkin diagram of $G$ is simply laced, the cohomology groups involved are at most one dimensional, and thus the Bott--Samelson varieties constructed for $X$ are isomorphic to the ones constructed for $G/B$. From this, one obtains that $X$ and $G/B$ are isomorphic.

In the case of multiply laced diagrams the situation is more involved, since at certain steps the cohomology groups involved in the construction of Bott--Samelson varieties may have dimension bigger than one. In the case of diagrams of type $\DB$ or $\DC$ it is still possible to find {\it good words}, for which that issue does not occur, and one can conclude as in the  simply laced cases. At this point one is left with the Dynkin diagrams $\DF_4$ and $\DG_2$, which are completely different since, remarkably, they do not admit any good word.

The characterization of complete flag manifolds has been achieved also in these two cases, using vector bundles techniques in the case of $\DG_2$ (see \cite[Lemma 4.1]{OSWW}, \cite{Wa2}, \cite[Theorem 5]{MOSW}) and, in the case of $\DF_4$, by reducing it to the characterization of the rational homogeneous space $\DF_4(1)$ in terms of its VMRT, a problem that was solved by Mok in \cite{Mk3} by using techniques of complex analysis and differential geometry.

On the other hand, the fact that the characterization via Bott--Samelson varieties works smoothly for simply laced diagrams suggests the problem of understanding the role of the excess of parameters in the construction of Bott--Samelson varieties for multiply laced diagrams, and the deformations that these parameters provide. Degenerations of Bott--Samelson varieties are reasonably well-known (see \cite{GK94,PK16}); our question requires studying their deformations, a problem that, to our best knowledge, has not been considered yet in the literature. 

In a nutshell, the interpretation of the excess of parameters that we propose is the following: the lacing of a Dynkin diagram is closely related to the existence of a geometric (orthogonal or skew--symmetric) structure on one of the rational homogeneous varieties associated to it, say $G/P$; the complete flag $G/B$ is then constructed upon the corresponding notion of isotropy with respect to that structure. When restricted to certain subvarieties of $G/P$, images of some Bott--Samelson varieties, that geometric structure might have non trivial moduli, and each possible choice of an element in that moduli may give rise to a different successor in the construction of the tower of Bott--Samelson varieties. Two successors would be different as $\P^1$-bundles over their predecessor, while they may still be isomorphic as varieties; in this case we will say that the word $\el$ used for the construction of the Bott--Samelson varieties is {\em flag--compatible} (see Definition \ref{def:flag-compatible}).
In particular, flag--compatible  words for $\DF_4$ and $\DG_2$ can be used to characterize the corresponding complete flag manifolds, extending the main line of argumentation of \cite{OSWW,OSW} to these cases. 
In this paper we prove  the flag--compatibility of two words of maximal length for these diagrams, by interpreting geometrically the extra parameters for their Bott--Samelson varieties.  Combining this with  \cite[Propositions 4.6 and 4.8]{OSWW} we get the statement of Theorem \ref{thm:main2}:

\begin{main}%[Theorem \ref{thm:main2}]
\label{thm:main}
Let $\cD(D)$ be a complete flag manifold. Then there exists a reduced word of maximal length $\el$ flag--compatible with $\cD(D)$.
\end{main}
  
 As a corollary we get a unified and completely algebraic proof of the characterization of complete flag manifolds given in \cite[Theorem 1.2]{OSWW}.  Furthermore, our arguments reveal a close relation among the parameter spaces of Bott--Samelson varieties and the moduli of isotropy structures induced by the lacing of the diagram on Schubert varieties. The problem of establishing if this relation is an equivalence for every reduced word remains open; a negative answer in this direction could lead to new examples of (probably singular) varieties closely related to flag varieties via deformations. 
 
%The problem of establishing whether {\em every} reduced word for a given Dynkin diagram is flag--compatible, that has a positive answer in the case of rank two, is still open in general.

\bigskip
\noindent{\bf Outline of the paper: }
%The outline of the paper is the following: 
after presenting the notation with use for rational homogeneous varieties and their contractions (Section \ref{sec:notn}), we introduce in Section \ref{sec:BSvar} the recursive construction of Bott--Samelson varieties upon reduced words. %, and explain how they can be used to characterize complete flag manifolds. 
In particular we define the notion of flag--compatibility for reduced words, which allows us to state in a precise way the main result of the paper (Theorem \ref{thm:main}); we conclude the Section by recalling some descent rules (Proposition \ref{prop:descent}) that allow us to compute the set of parameters involved in the recursive construction of Bott--Samelson varieties. We deal with complete flag manifolds of Picard number two in Section \ref{sec:B2}, showing that the excess of parameters may be interpreted in terms of deformations of a contact form on $\P^3$ in the case $\DB_2$, and of a Cayley bundle %$\DG_2$-structure 
on the smooth $5$-dimensional quadric  in the case $\DG_2$. 

The rest of the paper is devoted to the case $\DF_4$, that is the only case in which a purely algebraic proof of the characterization of flag varieties given in \cite{OSWW,OSW} was not known. We start in Section \ref{sec:preliminaries} by describing geometrically the complete flag manifold $\DF_4(D)$ in terms of flags in the projective contact manifold $\DF_4(1)$, and in terms of isotropic relative flags on the $\P^5$-bundle $\DF_4(1,4)\to \DF_4(1)$. In Section \ref{sec:redwords} we describe some algorithms, constructed upon the descent rules previously introduced, that allow us to choose a geometrically meaningful reduced word in the case $\DF_4$; more concretely, we choose a word $\el$ in which the excess of parameters occurs at a unique step. In Section \ref{sec:PandB} we look at this excess of parameters from the point of view of the geometric description presented in Section \ref{sec:preliminaries}, relating it to automorphisms of a certain Schubert variety, isomorphic to a family of isotropic relative flags on a $\P^5$-bundle $P$ over a plane $\P^2\subset\DF_4(1)$. We then conclude the paper (Section \ref{sec:proof}) by showing that the extra parameters of the Bott--Samelson construction correspond to the possible isotropy structures on that $\P^5$-bundle.\par
\bigskip
\noindent{\bf Acknowledgements: } The authors would like to thank J. Wi\'sniewski for his interesting comments. This project started when the second author was a Visiting Researcher at the Department of Mathematics of the University of Warsaw; he would like to thank this institution for its support and hospitality.  We also thank E. Ballico for providing useful references.

%%%%%%%%%%%%%%%%%%%%%%%%%%%%%%%%%%%%%%%%%%%%%%
%%%%%%%%%%%%%%%%%%%%%%%%%%%%%%%%%%%%%%%%%%%%%%
\section{Notation}\label{sec:notn}

We will start by introducing some notation on rational homogeneous varieties; we refer the interested reader to \cite{MOSWW} for details.

Let $G$ be  a semisimple algebraic group with Lie algebra $\fg$, Dynkin diagram $\cD$ and Weyl group $W$. We fix a Cartan subgroup $H\subset G$ and a Borel subgroup $B\subset G$ containing $H$, determining a base of positive simple roots $\Delta$ of $G$. The rational homogeneous varieties of type $\cD$ are, by definition, the projective quotients of $G$; any of them is determined by a subset of the set of nodes $D$ of $\cD$. More concretely, given a nonempty subset $I\subset D$, and its complementary subset $I^c=D\setminus I$, we may  consider the corresponding set of reflections $\{s_j,\,\,j\in I^c\}\subset W$, and the subgroup $W(I^c)$ that they generate, and define a rational homogeneous variety as follows:
$$
\cD(I):=G/BW(I^c)B.
$$
We will represent this variety by the Dynkin diagram $\cD$ marked on the nodes of $I$.
 
Furthermore, every rational homogeneous variety of type $\cD$ can be constructed in this way. Note that the Picard number of $\cD(I)$ is equal to $\sharp(I)$, and that the variety $\cD(D)$ equals the complete flag manifold $G/B$ of type $\cD$. 

Moreover, an inclusion $I\subsetneq J\subset \Delta$ provides a smooth fiber type contraction $\pi_{J,J\setminus I}:\cD(J)\to \cD(I)$; in the case in which $J=D$ we will simply write $\pi_{I^c}:=\pi_{D,I^c}:\cD(D)\to \cD(I)$.  The fibers of $\pi_{J,J\setminus I}$ are isomorphic to the rational homogeneous manifold $\cD_{I^c}(J\setminus I)$, where $\cD_{I^c}$ denotes the Dynkin subdiagram of $\cD$ supported on the nodes indexed by $I^c$.

In order to simplify the notation, when using an explicit expression of a subset $I\subset D$, we will avoid using curly brackets, and also commas when the subset appears as a subindex. For instance, we will write: $\pi_{12}:\DF_4(D)\to\DF_4(3,4)$, instead of $\pi_{\{1,2\}}:\DF_4(D)\to\DF_4(\{3,4\})$. 
In the case in which $I^c$ consists of a unique node $j$, the contraction
$$
\pi_j:\cD(D)\to \cD(D\setminus\{j\})
$$
is a $\P^1$-bundle, whose relative canonical bundle is denoted by $K_j$. Denoting by $\Gamma_j$ a fiber of $\pi_j$ (and, by abuse of notation, also its numerical class), one may compute the intersection matrix $(-K_i\cdot\Gamma_j)$, which turns out to  be equal to the Cartan matrix of the Lie algebra $\fg$.

\begin{remark}\label{rem:Picgens}
The classes $\Gamma_i$ are the generators of the Mori cone of $\cD(D)$, which is  simplicial. Moreover, there exist line bundles $H_i$, $i\in D$, satisfying that:
$$
H_i\cdot\Gamma_j=\delta_{ij}.
$$
Then the numerical classes of the $H_i$'s are the equations of the facets of $\cNE{\cD(D)}$, 
and each line bundle $H_i$ is the pullback of an ample line bundle on $\cD(i)$, $i\in D$; abusing notation, we will denote it also by $H_i\in\Pic(\cD(i))$. Finally, since $H_i\cdot\Gamma_i=1$, it follows that $H_i\in\Pic(\cD(i))$ is the ample generator of $\Pic(\cD(i))$. Since the image into $\cD(i)$ of a curve in the class $\Gamma_i$ has $H_i$-degree equal to one, we call these curves {\em lines on $\cD(i)$}.  
\end{remark}

\section{Bott--Samelson varieties}\label{sec:BSvar}

In this section we will recall some basic facts on Bott--Samelson varieties, which 
appear classically in the study of Schubert cycles of homogeneous manifolds (we refer to \cite{LT} and to the references therein for details), and we give a precise formulation of the main result of the paper:

\begin{setup}\label{set:FTvar} Let $X = \cD(D)$ be a complete flag manifold  of type $\cD$ and denote by  $\pi_i:X\to X^i:=\cD(D\setminus \{i\})$, $i\in D=\{1,\dots,n\}$ the elementary contractions of $X$. For each $i\in D$ we denote by $K_i$ the relative canonical bundle of $\pi_i$, and by $\Gamma_i$ the numerical class of its fiber. \end{setup}

In the sequel, we will need to deal with finite sequences of elements of $D$, that we will call {\em words of }$D$. We will use the following notation: 

\begin{notation}\label{notn:LT} Given a word $\el=(l_1,\dots,l_r)$, $l_i\in D$, we set  $\el[r]:=\emptyset$ and, for any 
$0\leq s\leq r-1$, $\el[s]:=(l_1,\dots,l_{r-s})$.
\end{notation}

\begin{remark}\label{rem:length} Denoting by  $W$ the Weyl group of $G$, the {\em length} $\ell(\el)$ of the word $\el=(l_1,\dots,l_r)$ is the length  of the element $s_{l_1} \circ s_{l_2} \circ \dots \circ s_{l_r}$ of $W$, where $s_i$ is the simple reflection associated with the node $i$ of the corresponding Dynkin diagram. The word $\el$ is {\em reduced} if its length is equal to $r$. A reduced word $\el$ has maximal
length if it cannot be written as $\el = \el_0[k]$, with $\el_0$ reduced and $k > 0$. A reduced word has maximal length iff the corresponding element in the Weyl group is the (unique) longest element, whose length equals the dimension of $X$.
For these and other properties the length of an element in a Weyl group  see \cite[1.6-1.8]{Hum3}); in practice the computation of the length is already implemented as an algorithm in several computer softwares, such as Sage.
\end{remark}

Given a reduced word $(l_1,l_2,\dots)$ of $D$, Bott--Samelson varieties are the smooth counterparts of the loci of (ordered) chains of curves in the classes $\Gamma_{l_1},\Gamma_{l_2},$ and so on. More precisely:

\begin{definition}\label{def:BS} Given a word $\el=(l_1,\dots,l_r)$ of $D$ and a point $x\in X$, we define a sequence of smooth varieties $Z_{\el[s]}$, $s=0,\dots,r$, called the {\it Bott--Samelson varieties of $X$ associated with $\el$, with starting point $x$},
and three sequences of morphisms
$$f_{\el[s]}:Z_{\el[s]} \to X,\quad p_{\el[s+1]}:Z_{\el[s]}\to Z_{\el[s+1]},\quad \sigma_{\el[s+1]}:Z_{\el[s+1]}\to Z_{\el[s]},$$
defined recursively as follows: for $s=r$ we set $Z_{\el[r]}:=\{x\}$ and $f_{\el[r]}$ to be the inclusion of $x$ in $X$. Then, for $s<r$, we define $Z_{\el[s]}$, $f_{\el[s]}$,  $p_{\el[s]}$, and $\sigma_{\el[s]}$ by considering the composition
$g_{\el[s+1]}:=\pi_{l_{r-s}}\circ f_{\el[s+1]}:Z_{\el[s+1]}\to X^{l_{r-s}}$ and taking its fiber product with $\pi_{l_{r-s}}$:
$$
\xymatrix@=35pt{Z_{\el[s]}\ar[r]^{f_{\el[s]}}\ar[d]^{p_{\el[s+1]}}&X\ar[d]^{\pi_{l_{r-s}}} \\
Z_{\el[s+1]}\ar@/^/[u]^{\sigma_{\el[s+1]}} \ar[ru]_{f_{\el[s+1]}}\ar[r]^{g_{\el[s+1]}}&X^{l_{r-s}}}
$$
The universal property of the fiber product tells us that $f_{\el[s+1]}$  factor via $Z_{\el[s]}$, providing the section $\sigma_{\el[s+1]}$ of $p_{\el[s+1]}$ shown in the diagram. 
Given a Bott--Samelson variety $Z_{\el}$ of $X$ with starting point $x$, its image into $X$, denoted by:
$$
\Sc_\el:=f_{\el}(Z_\el)\subset X,
$$
is the {\em Schubert variety of $X$}, associated with the word $\el$ and the point $x$.
\end{definition}

Summing up, we have a diagram of the form:

$$
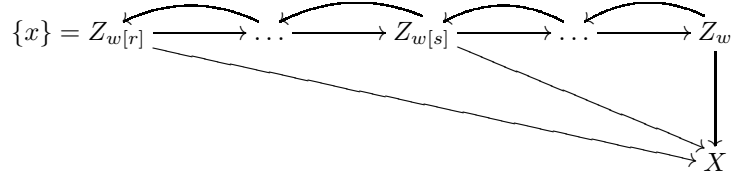
\xymatrix@=35pt{\{x\}=Z_{\el[r]}\ar[r]\ar[drrrr]&\dots\ar@/_/[]+<-1ex,1ex>;[l]+<4ex,1ex>\ar[r]&Z_{\el[s]}\ar@/_/[]+<0ex,1.5ex>;[l]+<1ex,1ex>\ar[r]\ar[drr]+<-1ex,1.5ex>&\dots\ar@/_/[]+<-1ex,1ex>;[l]+<2ex,1ex>\ar[r]&Z_{\el}\ar@/_/[]+<-1ex,1.5ex>;[l]+<1ex,1ex>\ar[d]
\\
&&&&X}
$$
where the left-headed arrows are the $\P^1$-bundles $p_{\el[s+1]}$, the right-headed arrows are sections 
$\sigma_{\el[s+1]}$, and the maps to $X$ are evaluations $f_{\el[s+1]}$, that commute with the sections. 

%\begin{definition}\label{def:length}
%A word $\el=(l_1,\dots,l_r)$ is said to be {\em reduced of length $r$} if and only if $\dim 
%\Sc_{\el[s]}
%= r-s$, for every $s$. We say that a reduced word $\el$ has {\em maximal length} if $\el$ cannot be written as $\el=\el'[k]$, with $\el'$ reduced and $k>0$. Note that, in the case in which $X$ is rationally chain connected by curves in the classes $\Gamma_j$, we may always find a reduced word of maximal length $r=\dim(X)$.
%\end{definition}

\begin{remark} The existence of the section $\sigma_{\el[s+1]}$ tells us that $Z_{\el[s]}$ can be described as the projectivization of an extension $\cF_{\el[s]}$ of $\cO_{Z_{\el[s+1]}}$ by $f_{\el[s+1]}^*K_{l_{r-s}}$:
\begin{equation}\label{eq:F}
0 \lra f_{\el[s+1]}^*K_{l_{r-s}}\longrightarrow
\cF_{\el[s]} \longrightarrow \cO_{Z_{\el[s+1]}}\lra 0.
\end{equation}
In particular $Z_{\el[s]}$ is completely determined by a cohomology class  $$\zeta_{\el[s]} \in H^1(Z_{\el[s+1]},f_{\el[s+1]}^*K_{l_{r-s}}),$$
modulo homotheties.
\end{remark}

%Let us denote by $\beta_{i(r-i)}$ the class in $N_1(Z_{\el[r-i]})$ of the fibers of $p_{\el[r-i+1]}:Z_{\el[r-i]} \to Z_{\el[r-i+1]}$. We  will denote by $\beta_{i(s)}$ the image of this class into $N_1(Z_{\el[s]})$, via push forward with the sections $\sigma_{\el[r-j]}$, $j= i, \dots, r-s-1$. Note that, by construction, $f_{\el[s]*}\beta_{i(s)}=[\Gamma_{l_i}]$. 

Assume that there exists $k>s$ satisfying that $l_{r-k}=l_{r-s}$ and let $C_k\cong\P^1$ in $Z_{\ell[s+1]}$ be the image of a fiber of $p_{\el[k+1]}:Z_{\el[k]}\to Z_{\el[k+1]}$  via the sections $\sigma_{\el[r-j]}$, $j= r-k, \dots, r-s-1$. 
Then (see \cite[Remark 3.4]{OSWW}) the restriction to $C_k$ of the sequence
 (\ref{eq:F}) is
$$0\lra  \cO_{\P^1}(-2)\longrightarrow \cO_{\P^1}(-1) \oplus \cO_{\P^1}(-1)
\longrightarrow \cO_{\P^1} \lra  0;$$
in particular the image of $\zeta_{\el[s]}$ via the restriction map 
\begin{equation}\label{eq:restriction}
r_k:=H^1(Z_{\el[s+1]},f_{\el[s+1]}^*K_{l_{r-s}}) \to H^1(C_k,f_{\el[s+1]}^*K_{l_{r-s}}|_{C_k})
\end{equation}
is not zero.

\begin{definition}
We will say that a cocycle $\zeta \in  H^1(Z_{\el[s+1]},f_{\el[s+1]}^*K_{l_{r-s}})$ is {\em non trivial on repeated curves} if it  is not contained in $\ker(r_k)$, for every $k > s$ such that $l_{r-k}=l_{r-s}$.
\end{definition}

\begin{definition}\label{def:flag-compatible}
We will say that a reduced word $\el$ is {\em flag--compatible with $\cD(D)$} if  any deformation of $Z_{\el[s]}$ given by a cocycle which is non trivial on repeated curves is isomorphic to $Z_{\el[s]}$, for any $s=0, \dots, \ell(\el)$. 
If $\el$ is flag--compatible with $\cD(D)$ we will also say that $Z_{\el}$ is {\em flag--compatible with $\cD(D)$}.

\end{definition}

We can now state the main result of the paper:

\begin{theorem}\label{thm:main2}%\marginpar{\tiny{\color{red} This is Thm. 1.1. Do we want to keep it here?}}
Let $X = \cD(D)$ be a complete flag manifold. Then there exists a reduced word of maximal length $\el$,  which is flag-compatible with $\cD(D)$.   
\end{theorem}

In  \cite[Propositions 4.6 and 4.8]{OSWW} it has been shown that:

\begin{itemize}
\item If $\cD$ is equal to $\DA_n$, $\DD_n$ or $\DE_n$, then  every reduced word $\el$ of $D$ is flag--compatible with $\cD(D)$.
\item If $\cD$ is equal to $\DB_n$ or $\DC_n$, there exists a reduced word $\el$ of maximal length in $D$ which is flag--compatible with $\cD(D)$.
\end{itemize}

In the present paper we are going to deal with the cases $\cD=\DF_4, \DG_2$; then, by \cite[Lemma 3.2]{OSWW} the result will hold for every complete flag manifold of a semisimple algebraic group.

\subsection{Descent rules for cohomology}

We will now describe a procedure to compute bounds on the dimension of the groups $H^1(Z_{\el[s+1]},f_{\el[s+1]}^*K_{l_{r-s}})$, by descending the Bott-Samelson tower one step at a time.

To simplify and compactify the notation, whenever it is clear which Bott--Samelson variety $Z_\el$ we are considering as a base, we will set, for a line bundle $L \in \Pic(X)$,  $$e^{L}:=f_{\el}^*L ,$$
and  
$$h^i(Z_\el, e^{L_1}+e^{L_2}):=h^i(Z_\el, e^{L_1})+h^i(Z_\el, e^{L_2}).$$
Using standard techniques of coherent subsheaves one may formulate the following set of {\em descent rules}, which is a straightforward corollary of \cite[Lemma 3.13]{OSWW}:

\begin{proposition}\label{prop:descent}
Let $\el=(l_1,\dots,l_r)$ be the word defining the Bott-Samelson variety $Z_\el$, and let $L$ be a line bundle on $X$, of degree $s$ with respect to $\Gamma_{l_r}$. Then:
\begin{enumerate}
%\item%[(DR2)] 
%If $s=-1$, then $h^i(Z_\el,e^L)=0$ for all $i$.
%[(DR3)] 
\item[(\Minus)] If $s\le-1$, then $h^i(Z_\el,e^{L})\le\sum_{j=1}^{-s-1}h^{i-1}(Z_{\el[1]},e^{L-jK_{l_r}})$, for all $i$; in particular, if $s=-1$, then $h^i(Z_\el,e^L)=0$ for all $i$.
\item[(0)]%[(DR1)] 
If $s=0$, then $h^i(Z_\el,e^{L})=h^i(Z_{\el[1]},e^{L})$ for all $i$.
%[(DR4)] 
\item[(\Plus)] If $s\geq 1$, then $h^i(Z_\el,e^L)\leq \sum_{j=0}^{s}h^i(Z_{\el[1]},e^{L+jK_{l_r}})$ for all $i$.
\end{enumerate}
\end{proposition}

\section{Taming bad words: the rank two cases}%the cases ${\DB_2}$ and ${\DG_2}$} 
\label{sec:B2}

An important point in the use of the descent rules in the proof of  \cite[Propositions 4.6 and 4.8]{OSWW} is that, at each step of the construction of Bott--Samelson varieties (say, in the construction of $Z_{\el[k-1]}$ upon $Z_{\el[k]}$, with $\el[k-1]=(l_1,l_2,\dots,j)$), flag--compatibility is guaranteed by a cohomological condition of the form:
$$h^1(Z_{\el[k]},e^{K_j})\leq 1.$$
The fact that in the cases $\DB_n$ and $\DC_n$ we may always find a good reduced word for which this condition is satisfied at each step (see \cite[Propositions 4.8]{OSWW}) does not mean that the corresponding flag varieties cannot be reconstructed by means of ``bad'' reduced words, in which some of the cohomology groups that we need to consider have dimension at least two. Furthermore, the problem  of interpreting this excess of extensions, and of checking whether the extra extensions provide isomorphic Bott--Samelson varieties or not, is fundamental in the cases $\DG_2$, and $\DF_4$, for which there is no possible choice of a ``good'' reduced word that allows us to avoid this problem. 

%As we already observed in the introduction, one may address these two exceptional cases by means of certain ad hoc arguments.   

In this section we illustrate the kind of interpretation we propose for the excess of extensions by considering the only bad word for the complete flag of type $\DB_2$. Then we will show how this interpretation can be extended to the case $\DG_2$, leading to a characterization of the corresponding flag manifold in terms of its Bott--Samelson varieties. The application of these ideas to the characterization of the complete flag manifold of type $\DF_4$ requires also a number of technicalities, with which we deal in the next sections.

\subsection{The $\DB_2$ case}\label{ssec:B2}

In the case of the complete flag of type $\DB_2$, we have a unique bad word, $\el=(2,1,2,1)$, for which we have a possible excess of extensions only at the third step:
\begin{eqnarray*}
h^1(Z_{(212)},e^{K_1}) &\le  &h^1(Z_{(21)},e^{K_1}) + h^1(Z_{(21)},e^{K_1+K_2})+ h^1(Z_{(21)},e^{K_1+2K_2})\\&=&1+0+h^1(Z_{(2)},e^{K_1+2K_2})=2.
\end{eqnarray*}

The only condition we have on the cocycle $\theta$ defining $Z_\el$ is that its restriction to a fiber  $\gamma_1$ of $Z_{(21)}\to Z_{(1)}$ is different from zero (cf \cite[Remark~3.4]{OSWW}), so that we only know that:
$$
\theta\in H^+(\gamma_1):=\{\theta'\in H^1(Z_{(212)},e^{K_1})\,\,|\,\, \theta'_{|\gamma_1}\neq 0\}\cong \C^*\times\C.
$$
Then the $\P^1$-bundles constructed by means of extensions in $H^+(\gamma_1)$ are param\-etr\-ized by the quotient $E^+(\gamma_1)$ of $H^+(\gamma_1)$ modulo homotheties. Summing up, the $\P^1$-bundle $Z_{(2121)}\to Z_{(212)}$ of Bott--Samelson varieties for $\DB_2$ deforms in a family parametrized by $E^+(\gamma_1)\cong\C$. 

In order to interpret this family, let us recall first some basic facts on this flag manifold. 
The manifolds $\DB_2(1)$ and $\DB_2(2)$ are, respectively, a smooth $3$-dimensional quadric $\Q^3$, and a $3$-dimensional projective space $\P^3$, and we have contractions
$$
\xymatrix{&\DB_2(1,2)\ar[dr]^{\pi_1}\ar[dl]_{\pi_2}&\\\DB_2(1)&&\DB_2(2)}
$$
We interpret this diagram as a family of lines in $\P^3$,  $\pi_2:\DB_2(1,2)\to\DB_2(1)$, with evaluation $\pi_1$. This family is not the complete family of lines in $\P^3$, but the family of isotropic lines in $\P^3$ with respect to a certain contact form on $\P^3$, that we will discuss later. 

The variety $Z_{(212)}$ is isomorphic to the corresponding Bott--Samelson variety of the complete flag of type $\DB_2$,   
and so it admits a surjective morphism 
$$\pi_2\circ f_{(212)}:Z_{(212)}\lra \DB_2(2)=\P^3.$$ 
One may check that this is in fact the composition of two blowups: first of a line $R\subset\P^3$ (the image of $Z_{(2)}$ in $\P^3$), and then of a conic in the exceptional divisor. The Bott--Samelson variety $Z_{(2121)}$ of the flag of type $\DB_2$ is then constructed by pulling back to $Z_{(212)}$ the bundle $\DB_2(1,2)\to\DB_2(2)$ via that morphism. 

This bundle can be described as the projectivization of a null correlation bundle $\cN$ on $\P^3$, which appears as the kernel of a contact form $\rho$ on $\P^3$:
$$
0\ra \cN(1)\lra T_{\P^3}\stackrel{\rho}{\lra}\cO_{\P^3}(2)\ra 0
$$
with respect to which the $\P^1$-bundle $Z_{(212)}\to Z_{(21)}$ may be thought of as the family of isotropic lines meeting $R$, which is isotropic itself. 

The key point here is that the contact form $\rho$ is not unique: the possible contact forms in $\P^3$ are in one to one correspondence to  nondegenerate antisymmetric $4\times 4$ matrices, and the corresponding bundles $\P(\cN)$ correspond to the classes modulo homotheties of those matrices. 

Now one may check that the family of bundles $\P(\cN)$ for which the line $R$, and the lines parametrized by $Z_{(212)}\to Z_{(21)}$ are isotropic is $1$-dimensional, parametrized by an affine line $\C$. In fact, by choosing an appropriate system of coordinates in $\P^3$, the antisymmetric matrices that define the bundles $\cN$ satisfying that property may be written as:
$$
\alpha\left(\begin{array}{cccc}0&0&0&1\\0&0&1&0\\0&-1&0&\mu\\-1&0&-\mu&0\end{array}\right),\quad \alpha\in\C^*,\mu\in \C.
$$
Moreover, given two different bundles $\P(\cN_1)$ and $\P(\cN_2)$, their pullbacks to  $Z_{(212)}$ are different as bundles over $Z_{(212)}$ (but not as varieties), and so we have an injective morphism:
$$
\C\lra E^+(\gamma_1)\cong\C.
$$ 
The fact that this map is indeed a morphism follows from the universal property of the universal family of extensions (cf. \cite{Lg83}). But the injectivity of the map implies its surjectivity %%% by purely topological reasons, (image will be constructible), hence equal to $\C$ minus some points
 and so we may claim that every bundle of the family $E^+(\gamma_1)$ corresponds to the choice of a contact form compatible with the family $Z_{(212)}\to Z_{(21)}$, and we conclude that any $\P^1$-bundle over $Z_{(212)}$ of the family $E^+(\gamma_1)$ is a Bott--Samelson variety for a flag manifold of type $\DB_2$, defined by a certain contact form on $\P^3$. 

\subsection{The $\DG_2$ case}\label{ssec:G2}

Let us start by recalling some geometric facts on the rational homogeneous varieties of type $\DG_2$, and the corresponding contractions:
$$
\xymatrix{&\DG_2(1,2)\ar[dr]^{\pi_1}\ar[dl]_{\pi_2}&\\\DG_2(1)&&\DG_2(2)}
$$
We will usually think of this as a subfamily of the complete family of lines in the $5$-dimensional quadric $\DG_2(1)$, more concretely as the subfamily of lines that are {\em isotropic} with respect to a structure of $\DG_2$-variety on this quadric; in the next section we will show how these structures are defined, and study their parameter space. We refer to \cite{O} for details.

\medskip

\subsubsection{$\DG_2$-structures on the smooth $5$-dimensional quadric}\label{ssec:G2struct}
%\begin{remark}\label{rem:G2}
The $5$-dimensional quadric $\DG_2(1)$  may be also written as the rational homogeneous space $\DB_3(1)$; as such, its family of planes is parametrized by the smooth $6$-dimensional quadric $\DB_3(3)$:
$$
\xymatrix{&\DB_3(1,3)\ar[dr]^{\pi'_1}\ar[dl]_{\pi'_3}&\\\DB_3(1)&&\DB_3(3)}
$$
 The evaluation morphism $\pi'_3:\DB_3(1,3)\to \DB_3(1)$ is a $\P^3$-bundle; more concretely, it is the projectivization of a rank $4$ globally generated vector bundle $\cS^\vee$ on $\DB_3(1)$, whose evaluation of global sections provides the morphism $\pi'_1$. A general section of $\cS^\vee$ (corresponding to a smooth hyperplane section of $\DB_3(3)$ in its natural embedding into $\P^7$) has no zeroes, providing a short exact sequence of vector bundles:
 $$
 \shse{\cO}{\cS^\vee}{\cG}
 $$
Moreover (see \cite[Theorem 2.8]{O2}),  we have a skew-symmetric isomorphism $\omega:\cS^\vee\to\cS(1)$, where $\cO(1)$ denotes the restriction to the $5$-dimensional quadric $\DB_3(1)$ of the hyperplane line bundle of its natural embedding into $\P^6$. Given a non vanishing section $s:\cO\to\cS^\vee$, its composition with $\omega$ factors via $\cG^\vee(1)$. The cokernel of the inclusion $\cO\to\cG^\vee(1)$ is a rank two vector bundle, that we denote by $\cC(1)$, which is, by construction, isomorphic to $\cC^\vee$, and whose projectivization is precisely the variety $\DG_2(1,2)$; moreover, the map $\pi_1:\DG_2(1,2)\to\DG_2(2)$ corresponds to the evaluation of global sections of $\cC(1)$. The bundle $\cC$ is usually called a {\em Cayley bundle} on $\DB_3(1)$. Following \cite{O}, any rank two bundle on the $5$-dimensional quadric with the same Chern classes as $\cC$ is constructed in this way. In other words, the possible structures of $\DB_3(1)$ as a $\DG_2$-variety are parametrized by the smooth hyperplane sections of $\DB_3(3) \subset \P^7$. 

Let us interpret geometrically the information above, in terms of families of lines and planes in the $5$-dimensional quadric $\DB_3(1)$. First of all, we have a commutative diagram with exact rows and columns:
$$
\xymatrix@=42pt{\cO\hspace{0.7cm}\ar@{^{(}->}[]+<-0.2ex,0ex>;[r]+<-1.3ex,0ex>^{\omega\circ s}\ar@{=}+<-2.4ex,-1.6ex>;[d]+<-2.4ex,1.8ex>&\hspace{0.7cm}\cG^\vee(1)\ar[r]\ar@{^{(}->}+<1.9ex,-2.2ex>;[d]+<1.9ex,1.8ex>&\cC(1)\cong\cC^\vee\ar@{^{(}->}+<2.6ex,-2.2ex>;[d]+<2.6ex,1.8ex>\\
\cO\hspace{0.7cm}\ar@{^{(}->}+<-0.2ex,0ex>;[r]^{s}&\cS^\vee\cong\cS(1)\ar[r]+<1ex,0ex>\ar+<1.9ex,-2.2ex>;[d]+<1.9ex,1.8ex>_{s^\vee}&\hspace{0.8cm}\cG\ar+<2.6ex,-2.2ex>;[d]+<2.6ex,1.8ex>_{(\omega\circ s)^\vee}\\
&\hspace{0.82cm}\cO(1)\ar@{=}[r]+<-0.4ex,0ex>&\hspace{0.8cm}\cO(1)
}
$$
Projectivizing the diagram above, we get a projective subbundle $\P(\cG)\subset \P(\cS^\vee)$, together with a section $\sigma_s:\DB_3(1)\lra \P(\cG)\hookrightarrow \P(\cS^\vee)$, whose composition with the morphism $\pi'_1: \P(\cS^\vee)\to \DB_3(3)$, that we denote by $$
\pi_s:\DB_3(1)\lra \DB_3(3),
$$ 
is given by the evaluation of global sections of $\cO(1)$, and so it is the embedding of the $5$-dimensional quadric $\DB_3(1)$ as a smooth hyperplane section of the $6$-dimensional quadric $\DB_3(3)$. The variety $\P(\cG)$ can then be thought as a family of planes in $\DB_3(1)$, that we call {\em isotropic with respect to $s$}, and that are parametrized by $\DB_3(1)$ via the map $\pi_s$. On the other hand, we have a rational map $$
\rho_s:\P(\cG)\dashrightarrow \P(\cC^\vee),
$$ 
that can be seen as the linear projection of the $\P^2$-bundle $\P(\cG)$ from the section $\sigma_s(\DB_3(1))$. In other words, given a point $p\in \DB_3(1)$, a line passing by $p$ is isotropic (with respect to the $\DG_2$-structure determined by the section $s$) if and only if it is contained in the plane $\pi_s(p)$. In particular, by using the map $\pi_s$, we may identify $\P(\cG)$ with  the subset of $\DB_3(1)\times \DB_3(1)$:
$$
\left\{(p_1,p_2)\in \DB_3(1)\times \DB_3(1)|\,\,p_1,p_2\in\ell\mbox{, for some $\ell$ isotropic w.r.t. }s\right\},
$$
by sending $(p_1,p_2)$ to the pair $(p_1,\pi_s(p_2))$. %; in fact the above argument tells us that $p_1\in\pi_s(p_2)$ if and only if $p_1$ lies on an isotropic line passing by $p_2$.  
Note that, via this identification, the section $\sigma_s$ becomes the diagonal morphism from $\DB_3(1)\to \DB_3(1)^2$, and the rational map $\rho_s:\P(\cG)\dashrightarrow \P(\cC^\vee)$ turns into the map sending:
$$
(p_1,p_2)\mapsto(p_1,p_1+p_2),
$$
which is defined in the complementary set of $\sigma_s(\DB_3(1))$.
%\qed
%\end{remark}

\medskip

\subsubsection{Characterizing $\DG_2(1,2)$ with Bott-Samelson varieties}\label{ssec:BSG2}

Let us show now how the geometric information provided above may be used to interpret the excess of extensions in the Bott--Samelson construction for the case $\DG_2$. In this case, the Cartan matrix is:
$$
(-K_i\cdot\Gamma_j)=\begin{pmatrix}2&-1\\-3&2\end{pmatrix}, 
$$
and we have two possible words of maximal length, both of them giving an excess of extensions at a certain step of the construction of the corresponding chains of Bott--Samelson varieties. We will consider here only the case $\el=(2,1,2,1,2,1)$; the other one can be treated in a similar way. 

By using Proposition \ref{prop:descent}, one finds possible excess of extensions at precisely one step, namely:
\begin{eqnarray*}
%h^1(Z_{(121)},e^{K_2})=\\[2mm]h^1(Z_{(12)},e^{K_2}) + h^1(Z_{(12)},e^{K_1+K_2})+ h^1(Z_{(12)},e^{2K_1+K_2})+ h^1(Z_{(12)},e^{3K_1+K_2})=\\[2mm]
%h^0(Z_1,\cO)+0+h^1(Z_1,e^{2K_1+K_2})+h^1(Z_1,e^{3K_1+K_2})+h^1(Z_1,e^{3K_1+2K_2})=\\[2mm]1+0+2+0=3;\\[2mm]
h^1(Z_{\el[2]},e^{K_2}) &= &h^1(Z_{(2121)},e^{K_2})\\[2mm] 
&\le& h^1(Z_{(212)},e^{K_2}+ e^{K_1+K_2}+e^{2K_1+K_2}+ e^{3K_1+K_2})\\[2mm]
&\le&h^0(Z_{(21)},e^\cO + e^{2K_1+K_2}+ e^{3K_1+K_2}+ e^{3K_1+2K_2})\\[2mm]
&\le&1+h^0(Z_{(2)},e^{2K_1+K_2} + e^{K_1+K_2}+ e^{3K_1+2K_2})
\\[2mm]
&=&1+1 + 0 +0=2.
\end{eqnarray*}

As in the case $\DB_2$ the only condition we impose on the cocycle $\theta$ defining $Z_\el$ is that its restrictions to the curves $\gamma_2$, fibers of  $Z_{(2)}\to \{x\}$ and of  $Z_{(212)}\to Z_{(21)}$ are different from zero; that is, for each of the curves $\gamma_2$ % (cf \cite[Remark~3.4]{OSWW}), so that we only know that:
$$
\theta\in %H^+:=\bigcap_{\gamma_2} H^+(\gamma_2), \quad 
H^+(\gamma_2):=\{\theta'\in H^1(Z_{(2121)},e^{K_2})\,\,|\,\, \theta'_{|\gamma_2}\neq 0\},%\cong \C^*\times\C,
$$
so that the $\P^1$-bundles constructed by means of extensions in $H^+(\gamma_2)$ are param\-etr\-ized by the quotient $E^+(\gamma_2)$ of $H^+(\gamma_2)$ modulo homotheties. %Summing up, the $\P^1$-bundle $Z_{(2121)}\to Z_{(212)}$ of Bott--Samelson varieties for $\DB_2$ deforms on a family parametrized by $E^+(\gamma_1)\cong\C$.

Denoting by $s\in H^0(\DB_3(1),\cS^\vee)$  the section defining the initial $\DG_2$-structure on $\DB_3(1)$, we will identify $H^+(\gamma_2)$ with the set of sections $s'\in H^0(\DB_3(1),\cS^\vee)$ that are compatible with $Z_{(2121)}$, in the sense that the flags point-line in $\DB_3(1)$ parametrized by $Z_{(2121)}$ are isotropic with respect to those structures.

Set $p:=\pi_2(Z_{(2)})$, and note that, as discussed in Section \ref{ssec:G2struct}, the variety $Z_{(21)}\subset \DG_2(1)$ is equal to:
$$
Z_{(21)}=\left\{(p_1,\ell)\in \DG_2(1,2)|\,\,p_1\in\ell,\,\,\ell\mbox{ isotropic line passing by }p\right\}.
$$
In a similar way, we may write:
$$
Z_{(2121)}=\left\{(p_2,\ell_1)\in \DG_2(1,2)|\,\, \ell_1\cap\ell\neq \emptyset, \mbox{ for some }\ell\mbox{ isotropic line passing by }p\right\}.
$$

The closures of the inverse images of these two sets in $\P(\cG)$ are then:
$$\begin{array}{l}\vspace{0.2cm}
\ol{\rho_s^{-1}(Z_{(21)})}=\left\{(p_1,p_2)\in \DB_3(1)^2|\,\, p_1,p_2\in\ell, \,\,\ell\mbox{ isotropic line by }p\right\},\\
\ol{\rho_s^{-1}(Z_{(2121)})}=\left\{(p_2,p_3)\in \DB_3(1)^2|\,\, \ell\cap(p_2+p_3)\neq\emptyset,  \,\,\ell\mbox{ isotropic line by } p\right\}.
\end{array}
$$
We conclude that the variety $Z_{(2121)}$ determines a set of isotropic planes equal to the image via $\pi_s:\DB_3(1)\to \DB_3(3)$ of $\pi'_1(\ol{\rho_s^{-1}(Z_{(2121)})}) \subset \DB_3(1)$. 

This is precisely the locus of chains of length two of lines isotropic with respect to $s$, which we claim is a $4$-dimensional quadric cone with vertex $p$, that is the intersection of the $5$-dimensional quadric $\DB_3(1)$ with its projective tangent hyperplane at $p$. In fact, given a general chain $(\ell,\ell_1)$ of two isotropic lines, with $p\in\ell$ and $\ell\neq\ell_1$, setting $p_1=\ell\cap \ell_1$, we get $p\in\ell\in\pi_s(p_1)$, and this tells us that $\ell_1\subset\pi_s(p_1)$ lies in the $4$-dimensional cone spanned by the lines in $\DB_3(1)$ passing by $p$. Since by construction $\pi'_1(\ol{\rho_s^{-1}(Z_{(2121)})})$ is $4$-dimensional, the claim follows. 

We may then describe the set of $\DG_2$-structures in $\DB_3(1)$ which are compatible with the chosen Bott-Samelson variety $Z_{(2121)}$:

\begin{lemma}
The set $S$ of smooth hyperplane sections of the $6$-dimensional quadric $\DB_3(3)$ whose corresponding $\DG_2$ structures in $\DB_3(1)$ are compatible with the chosen Bott-Samelson variety $Z_{(2121)}$, is isomorphic to $\C$.
\end{lemma}

\begin{proof}
We have already seen that the required sections are precisely those that contain the image via $\pi_s$ of the intersection of $\DB_3(1)$ with its tangent hyperplane at $p$. In the natural embedding of $\DB_3(3)$ in a $7$-dimensional projective space, they correspond precisely to the hyperplanes, different from $T_{\DB_3(3),p}$,  containing the $5$-dimensional projective space $T_{\DB_3(1),p}$. In particular, they are parametrized by an affine line $\C$.
\end{proof}

For every element $s'\in S$ we have a $\DG_2$-structure in $\DB_3(1)$, a family of flags in $\DB_3(1)$, that we denote $\DG_2(1,2)_{s'}$, providing a $\P^1$-bundle on $Z_{(2121)}$ obtained via the fiber product:
$$
\xymatrix{Z_{(21212)_{s'}}\ar[r]\ar[d]&\DG_2(1,2)_{s'}\ar[d]\\
Z_{(2121)}\ar[r]&\DB_3(1)
}
$$ 
By construction, every $Z_{(21212)_{s'}}$ is given by an element of $E^+(\gamma_2)$, hence, 
from the universal property of the universal family of extensions (cf. \cite{Lg83}) we have a morphism
$$
\psi:S \simeq \C\lra E^+(\gamma_2)\cong\C^k, \quad k \le 0.
$$

\begin{proposition}\label{prop:G2surj}
The map $\psi:S\to E^+(\gamma_2)$ is surjective.
\end{proposition} 

\begin{proof} 
As in the $\DB_2$ case it is enough to show that  $\psi$ is injective. 
Let $s'_1,s'_2\in S$ be two different elements, and let $H_1,H_2$ be the corresponding hyperplane sections of $\DB_3(3)$, whose intersection is the $4$-dimensional quadric cone $\pi_s(C)$, with $C=T_{\DB_3(1),p}\cap \DB_3(1)$. We will show that they provide different $\P^1$-bundles on $Z_{(2121)}$ by interpreting them as two different families of flags point-line in $\DB_3(1)$. 

Since, for every $i=1,2$, the families of $s'_i$-isotropic flags passing by a point $p'$ are contained in the $s'_i$-isotropic plane $\pi_{s'_i}(p')$, it is enough to show that for a general point $p'$ of the  cone $C$ we have  $\pi_{s'_1}(p')\neq\pi_{s'_2}(p')$. And for this it is enough to note that $\pi_{s'_i}(p')\in H_i\setminus \pi_s(C)$. 
\end{proof}

\begin{corollary}\label{cor:G2flagcomp}
The word $\el=(212121)$ is flag--compatible %for the flag manifold 
with $\DG_2(1,2)$.
\end{corollary}

%%%%%%%%%%%%%%%%%%%%%%%%%%%%%%%%%%%%%%%%%%%%%%

\section{Geometry of rational homogeneous varieties of type $\DF_4$}\label{sec:preliminaries}

Along the rest of the paper we will consider the rational homogeneous manifolds of type $\DF_4$. Let us recall that the Lie algebra $\ff$ has dimension $52$ and rank $4$, and that it has a unique associated semisimple algebraic group $\DF_4$, which is simply connected.
We will use the notation introduced in Section \ref{sec:notn}. The set of nodes of the Dynkin diagram of $\DF_4$, ordered from left to right, is $D=\{1,2,3,4\}$
$$
%%%%%%%%%%%%%%%
\ifx\du\undefined
  \newlength{\du}
\fi
\setlength{\du}{3.3\unitlength}
\begin{tikzpicture}
\pgftransformxscale{1.000000}
\pgftransformyscale{1.000000}

%%%%%% COLORS
\definecolor{dialinecolor}{rgb}{0.000000, 0.000000, 0.000000} % EXTERIOR
\pgfsetstrokecolor{dialinecolor}
\definecolor{dialinecolor}{rgb}{0.000000, 0.000000, 0.000000} % INTERIOR
\pgfsetfillcolor{dialinecolor}

%%%%%% NODES

\pgfsetlinewidth{0.300000\du}
\pgfsetdash{}{0pt}
\pgfsetdash{}{0pt}
%\pgfsetmiterjoin

%%% #1
\pgfpathellipse{\pgfpoint{-6\du}{0\du}}{\pgfpoint{1\du}{0\du}}{\pgfpoint{0\du}{1\du}}
\pgfusepath{stroke}
\node at (-6\du,0\du){};
%\pgfpathellipse{\pgfpoint{-6\du}{0\du}}{\pgfpoint{1\du}{0\du}}{\pgfpoint{0\du}{1\du}}
%\pgfusepath{fill}
%\node at (-6\du,0\du){};

%%% #2
\pgfpathellipse{\pgfpoint{4\du}{0\du}}{\pgfpoint{1\du}{0\du}}{\pgfpoint{0\du}{1\du}}
\pgfusepath{stroke}
\node at (4\du,0\du){};
%\pgfpathellipse{\pgfpoint{4\du}{0\du}}{\pgfpoint{1\du}{0\du}}{\pgfpoint{0\du}{1\du}}
%\pgfusepath{fill}
%\node at (4\du,0\du){};

%%% #3
\pgfpathellipse{\pgfpoint{14\du}{0\du}}{\pgfpoint{1\du}{0\du}}{\pgfpoint{0\du}{1\du}}
\pgfusepath{stroke}
\node at (14\du,0\du){};
%\pgfpathellipse{\pgfpoint{14\du}{0\du}}{\pgfpoint{1\du}{0\du}}{\pgfpoint{0\du}{1\du}}
%\pgfusepath{fill}
%\node at (14\du,0\du){};

%%% #4
\pgfpathellipse{\pgfpoint{24\du}{0\du}}{\pgfpoint{1\du}{0\du}}{\pgfpoint{0\du}{1\du}}
\pgfusepath{stroke}
\node at (24\du,0\du){};
%\pgfpathellipse{\pgfpoint{24\du}{0\du}}{\pgfpoint{1\du}{0\du}}{\pgfpoint{0\du}{1\du}}
%\pgfusepath{fill}
%\node at (24\du,0\du){};

%%%%%% LINKS
\pgfsetlinewidth{0.300000\du}
\pgfsetdash{}{0pt}
\pgfsetdash{}{0pt}
\pgfsetbuttcap

{\draw (-5\du,0\du)--(3\du,0\du);}
%{\draw (5\du,0\du)--(13\du,0\du);}
{\draw (15\du,0\du)--(23\du,0\du);}
%{\draw (25\du,0\du)--(33\du,0\du);}
{\draw (4.65\du,0.7\du)--(13.35\du,0.7\du);}
{\draw (4.65\du,-0.7\du)--(13.35\du,-0.7\du);}

%%%%%% ARROW HEAD

{\pgfsetcornersarced{\pgfpoint{0.300000\du}{0.300000\du}}\definecolor{dialinecolor}{rgb}{0.000000, 0.000000, 0.000000}
\pgfsetstrokecolor{dialinecolor}
\draw (7\du,-1.2\du)--(10.8\du,0\du)--(7\du,1.2\du);}

%%%%%% TAGS
\node[anchor=west] at (28\du,0\du){${\rm F}_4$};

\node[anchor=south] at (-6\du,1.1\du){$\scriptstyle 1$};

\node[anchor=south] at (4\du,1.1\du){$\scriptstyle 2$};

\node[anchor=south] at (14\du,1.1\du){$\scriptstyle 3$};

\node[anchor=south] at (24\du,1.1\du){$\scriptstyle 4$};

\end{tikzpicture} 
%%%%%%%%%%%%%%%
$$
With this ordering, the corresponding Cartan matrix is equal to:
$$
(-K_i\cdot\Gamma_j)=\left(\begin{array}{cccc}2&-1&0&0\\-1&2&-2&0\\0&-1&2&-1\\0&0&-1&2\end{array}\right).
$$

Later on we will pay special attention to the the variety $\DF_4(1)$, which is the homogeneous contact manifold of type $\DF_4$; it has Picard number one and dimension $15$.  We will look at the rest of homogeneous manifolds of type $\DF_4$ in terms of their relation with respect to $\DF_4(1)$. 

\subsection{Vertical isotropic spaces}\label{ssec:vertiso}

In this section we will describe the contractions  
$$\pi_{234}:\DF_4(D)\to \DF_4(1),\quad \pi_{14,4}:\DF_4(1,4)\to\DF_4(1),$$
interpreting every fiber $\pi_{234}^{-1}(x)$ as a variety of isotropic flags of the $5$-dimensional projective space $\pi_{14,4}^{-1}(x)$, that is as a complete flag manifold of type $\DC_3$. Let us start by proving the existence of a vertical skew-symmetric form, with respect to which we will define a notion of isotropy: 

\begin{lemma}\label{lem:rank6bdl}
There exists a vector bundle $\cE$ of rank $6$ on $\DF_4(1)$ and a skew-symmetric isomorphism $\cE\cong\cE^\vee\otimes H_1$ whose projectivization is isomorphic to $\DF_4(1,4)$, so that its contraction to $\DF_4(1)$ is the natural projection. 
\end{lemma}

\begin{proof}
Note first that the fibers of $\DF_4(1,4)\to \DF_4(1)$ are $\P^5$'s, on which the images of the curves in the class $\Gamma_4$ are lines. In particular, the  pullback of the ample generator $H_4$ of $\Pic(\DF_4(4))$  is unisecant (see Remark \ref{rem:Picgens}), and then $\DF_4(1,4)$ is the projectivization of the rank $6$ vector bundle, $\cE:=\pi_*(H_4)$. The fact that $\DF_4(D)\to \DF_4(1)$ is a $\DC_3$-bundle implies that the vector bundle $\cE$ admits an everywhere non degenerate skew-symmetric form on each fiber, given by an isomorphism:
$$
\eta:\cE\lra\cE^\vee\otimes\cL,
$$
for a certain line bundle $\cL$ in $\DF_4(1)$. In order to compute $\cL$, we will consider the restriction to a line $\ell$ in $\DF_{4}(1)$, and use the interpretation of Grothendieck theorem for flag bundles over $\P^1$ that we have established in \cite[Section 3.3]{OSW}: this tells us that the the restriction of $\DF_4(D)\to\DF_4(1)$ to $\ell$ is determined by the following tagged Dynkin diagram of type  $\DC_3$: 
$$
%%%%%%%%%%%%%%%%
\ifx\du\undefined
  \newlength{\du}
\fi
\setlength{\du}{3.3\unitlength}
\begin{tikzpicture}
\pgftransformxscale{1.000000}
\pgftransformyscale{1.000000}

%%%%%% COLORS
\definecolor{dialinecolor}{rgb}{0.000000, 0.000000, 0.000000} % EXTERIOR
\pgfsetstrokecolor{dialinecolor}
\definecolor{dialinecolor}{rgb}{0.000000, 0.000000, 0.000000} % INTERIOR
\pgfsetfillcolor{dialinecolor}

%%%%%% NODES

\pgfsetlinewidth{0.300000\du}
\pgfsetdash{}{0pt}
\pgfsetdash{}{0pt}
%\pgfsetmiterjoin

%%% #1
%\pgfpathellipse{\pgfpoint{-6\du}{0\du}}{\pgfpoint{1\du}{0\du}}{\pgfpoint{0\du}{1\du}}
%\pgfusepath{stroke}
%\node at (-6\du,0\du){};
\pgfpathellipse{\pgfpoint{6\du}{0\du}}{\pgfpoint{1\du}{0\du}}{\pgfpoint{0\du}{1\du}}
\pgfusepath{fill}
\node at (6\du,0\du){};

%%% #2
\pgfpathellipse{\pgfpoint{-4\du}{0\du}}{\pgfpoint{1\du}{0\du}}{\pgfpoint{0\du}{1\du}}
\pgfusepath{stroke}
\node at (-4\du,0\du){};
%\pgfpathellipse{\pgfpoint{4\du}{0\du}}{\pgfpoint{1\du}{0\du}}{\pgfpoint{0\du}{1\du}}
%\pgfusepath{fill}
%\node at (4\du,0\du){};

%%% #3
\pgfpathellipse{\pgfpoint{-14\du}{0\du}}{\pgfpoint{1\du}{0\du}}{\pgfpoint{0\du}{1\du}}
\pgfusepath{stroke}
\node at (-14\du,0\du){};
%\pgfpathellipse{\pgfpoint{14\du}{0\du}}{\pgfpoint{1\du}{0\du}}{\pgfpoint{0\du}{1\du}}
%\pgfusepath{fill}
%\node at (14\du,0\du){};

%%%%%% LINKS
\pgfsetlinewidth{0.300000\du}
\pgfsetdash{}{0pt}
\pgfsetdash{}{0pt}
\pgfsetbuttcap

{\draw (5\du,0\du)--(-3\du,0\du);}
{\draw (-4.65\du,0.7\du)--(-13.35\du,0.7\du);}
{\draw (-4.65\du,-0.7\du)--(-13.35\du,-0.7\du);}

%%%%%% ARROW HEAD

{\pgfsetcornersarced{\pgfpoint{0.300000\du}{0.300000\du}}\definecolor{dialinecolor}{rgb}{0.000000, 0.000000, 0.000000}
\pgfsetstrokecolor{dialinecolor}
\draw (-10.8\du,-1.2\du)--(-7\du,0\du)--(-10.8\du,1.2\du);}

%%%%%% TAGS
%\node[anchor=west] at (18\du,0\du){${\rm B}_3$};

\node[anchor=south] at (6\du,1.1\du){$\scriptstyle 0$};

\node[anchor=south] at (-4\du,1.1\du){$\scriptstyle 0$};

\node[anchor=south] at (-14\du,1.1\du){$\scriptstyle 1$};

\end{tikzpicture} 
%%%%%%%%%%%%%%%%
$$
Following \cite[Example 3.19]{OSW}, this implies that its splitting type is of the form $(a^3,(a+1)^3)$, for a certain integer $a$. Note that a minimal section of $\P(\cE)$ over $\ell$ is given by a quotient $\cE_{|\ell}\to \cO_\ell(a)$. Since on the other hand we know that these minimal sections are images of curves in the class $\Gamma_1$ in $\DF_4(D)$, we have $a=H_4\cdot\Gamma_1=0$. We conclude that $\cL=H_1.$
\end{proof}

\begin{definition}\label{def:isotr}
Given a morphism $f:Y\to \DF_4(1)$ from a variety $Y$, the above isomorphism extends to an skew-symmetric isomorphism $\eta:f^*\cE\to f^*\cE\otimes f^*H_1$. In particular given a projective subbundle of $\P(\cF)\subset\P(f^*\cE)$, given by a surjection $f^*\cE\to\cF$ onto a vector bundle $\cF$ on $Y$, with kernel $\cG$, we may define its {\em orthogonal subbundle with respect to $\eta$} by:
$$\P(\cF)^{\perp}:=\P(\cG^\vee(1)),
$$
considered as a projective subbundle of $\P(\cE)$ via the surjection $$\cE\stackrel{\eta}{\lra}\cE^{\vee}(1)\lra\cG^\vee(1).$$ 
If $\P(\cF)\subseteq\P(\cF)^\perp$, we say that the bundle $\P(\cF)$ is {\em isotropic with respect to $\eta$}. In the case in which $f$ is the inclusion of a closed point $x$ of $\DF_4(1)$, we will say that $\P(\cF)$ is an {\em isotropic subspace} of the fiber of $\DF_4(1,4)\to \DF_4(1)$ over $x$. 
\end{definition}

Now we may interpret the fiber $\pi_{234}^{-1}(x)$ of $\DF_4(D)\to\DF_4(1)$ over a point $x$ as the variety of {\em flags of isotropic subspaces} in $\pi_{14,4}^{-1}(x)\cong\P^5$: 
$$
\pi_{234}^{-1}(x)=\left\{(p,\ell,\pi)\left|\,\,\begin{array}{l}p,\ell,\pi\subset\pi_{14,4}^{-1}(x)\mbox{ proj. subspaces of $\dim$. $0,1,2$, resp.}\\p\in\ell\subset\pi=\pi^\perp\subset\ell^\perp\subset p^\perp\end{array}\right.\right\}.
$$
This variety of flags is isomorphic to a complete flag manifold of type $\DC_3$.
Furthermore, we may describe explicitly the fibers of the elementary contractions $\pi_i$, $i=2,3,4$, as curves in the fibers $\pi_{234}^{-1}(x)$:
\begin{itemize}
\item A fiber $\gamma_2$ of $\pi_2$ contained in $\pi_{234}^{-1}(x)$ is determined by the choice of an isotropic line $\ell_0\subset\ell_0^\perp$ in $\pi_{14,4}^{-1}(x)$ and a point $p_0\in\ell_0$:
$$
\gamma_2=\left\{(p_0,\ell_0,\pi)\left|\,\,p_0\in\ell_0\subset\pi=\pi^\perp\right.\right\}.
$$
\item A fiber $\gamma_3$ of $\pi_3$  contained in $\pi_{234}^{-1}(x)$ is determined by the choice of an isotropic plane $\pi_0=\pi_0^\perp$  in $\pi_{14,4}^{-1}(x)$, and a point $p_0\in\pi_0$:
$$
\gamma_3=\left\{(p_0,\ell,\pi_0)\left|\,\,p_0\in\ell\subset\pi_0\right.\right\}.
$$
\item Finally, a fiber $\gamma_4$ is determined by an isotropic line $\ell_0\subset\ell_0^\perp$ and an isotropic plane $\pi_0=\pi_0^\perp$ containing $\ell_0$:
$$
\gamma_4=\left\{(p,\ell_0,\pi_0)\left|\,\,p\in\ell_0\subset\pi_0\right.\right\}.
$$
\end{itemize}   

\begin{remark}\label{rem:isolines}
Note that the curves of type $\gamma_4$ get mapped to lines in the fibers of $\DF_4(1,4)\to\DF_4(1)$. Their images are parametrized by the variety $\DF_4(1,3)$, so that the corresponding universal family and evaluation morphism are:
$$
\xymatrix{&\DF_{4}(1,3,4)\ar[dr]\ar[dl]&\\\DF_{4}(1,3)&&\DF_{4}(1,4)}
$$  
In this way, we will call $\DF_{4}(1,3,4)\to\DF_4(1,3)$ the {\em universal family of isotropic lines in $\DF_4(1,4)\to\DF_4(1)$}. 
\end{remark}

\subsection{Flags on $\bm{\DF_4(1)}$}\label{ssec:flagsF41}

Later on, we will need to think of the varieties $\DF_4(i)$ as parameter spaces of subvarieties of $\DF_4(1)$.
By looking at the fibers of the contractions of the varieties of the form $\DF_4(1,i)$, one may easily interpret the rest of varieties of type $\DF_4$ of Picard number one, as families of subvarieties in $\DF_4(1)$:
\begin{itemize}
\item $\DF_4(2)$ parametrizes lines in $\DF_4(1)$; furthermore, one may prove that it parametrizes all the lines in $\DF_4(1)$, and the subfamily of lines passing by a given point is isomorphic to a variety of type $\DC_3(3)\cong\Q^6$ (see \cite[Theorems 4.3 and 4.8]{LM})
\item $\DF_4(3)$ parametrizes a family of planes in $\DF_4(1)$, so that there the subfamily of planes passing by a given point is isomorphic to a variety $\DC_3(2)$ (a Lagrangian Grassmannian of isotropic lines in $\P^5$), and the subfamily of planes containing a line is isomorphic to $\P^2$.   
\item $\DF_4(4)$ parametrizes a family of smooth $5$-dimensional quadrics $\Q^5\cong\DB_3(1)$ contained in $\DF_4(1)$. As in the previous case, we may identify the subfamilies containing a given point --$\P^5\cong\DC_3(1)$--, a given line --$\P^2$-- or a plane of the family $\DF_4(3)$; in this last case, for instance, looking at the contractions 
$$
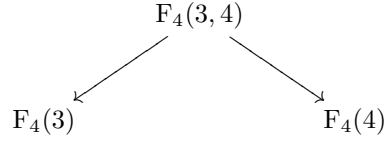
\xymatrix{&\DF_{4}(3,4)\ar[dr]\ar[dl]&\\\DF_{4}(3)&&\DF_{4}(4)}
$$
we may claim that the subfamily in question is parametrized by the image into $\DF_4(4)$ of a fiber of the contraction $\DF_4(3,4)\to\DF_4(3)$, which is a line in $\DF_4(4)$, image of a curve of numerical class $\Gamma_4$.  
\end{itemize}

\section{Reduced words of maximal length for ${\DF_4}$}\label{sec:redwords}

%Along the rest of the paper we will use the following notation:
%
%\begin{notation}\label{set:4}
%Let $X$ be the complete flag of type $\DF_4$, $X=\DF_4(D)$, $D=\{1,2,3,4\}$. For every $i\in D=\{1,2,3,4\}$, we denote by $\pi_i:X\to \cD(D\setminus\{i\})$ the corresponding elementary contraction, by $K_i$ its relative canonical bundle of $\pi_i$, and by $\Gamma_i$ the numerical class of its fiber. 
%\end{notation}

As in the cases $\DB_2, \DG_2$ (Section \ref{sec:B2}), we will consider an appropriately chosen reduced word $\el_0$ of maximal length for $\DF_4$, which will be shown to be flag--compatible in the following sections.
More specifically, in this section we will use the descent rules (see Proposition \ref{prop:descent})  --which we make more explicit for the case  $\DF_4$-- in order to show that, for this word, we have excess of extensions only at one step. 

We begin by observing that only the  groups of type $H^1(Z_{\el'}, e^{K_2})$ may give
 excess of extensions, in the following sense:

\begin{lemma}\label{lem:only2}
Let $\el$ be a reduced word of $\DF_4$,  with final letter different from $2$. Then, if $\el[1]$ is flag--compatible, also $\el$ is flag--compatible.
\end{lemma}

 \begin{proof}
If $\el=(l_1,\dots,l_{i}=:j)$ with $j=1,3,4$, then one has to check that $h^1(Z_{\el[s+1]},e^{K_j})\leq 1$ if there exists $k<i$ such that $l_k=l_i$, and $h^1(Z_{\el[s+1]},e^{K_j})=0$ otherwise. This follows from \cite[Corollary 4.5]{OSWW}.
 \end{proof}

In other words, we only need to take care of the steps of the Bott--Samelson construction in which the last letter of the word $\el[s]$ is equal to $2$. 
Unfortunately, one can check (by using Sage computer software, for instance) the following
\begin{fact}\label{fact:excess}
For every reduced word $\el$ of maximal length, there exists a subword $\el[s]$, finished in $j=2$, satisfying that $$h^1(Z_{\el[s+1]},e^{K_2})\geq 2.$$ 
\end{fact}
%Our first goal will be to choose a particular word of maximal length for $\DF_4$, for which the above feature occurs only in a particular step, in which we will be able to control and interpret geometrically the excess of extensions.

\subsection{Descent rules for $\bm{\DF_4}$}\label{ssec:descent}

In order to search efficiently within the set of reduced words of maximal length for $\DF_4$, we will start by rewriting our descent rules in the case of groups of the form $H^1(Z_{\el[1]},e^{K_2})$.  Let us start by introducing the following divisors (in the table the degrees are the degrees of the divisors on the curves $\Gamma_1, \dots, \Gamma_4$):\par
\medskip
\noindent
\def\arraystretch{1.2}
\scalebox{0.87}{
\begin{tabular}{|c|c|c||c|c|c|}
\hline
$e^L$&$L$&degrees&$e^{L'}$&$L'$&degrees\\\hline
$A$&$K_2$&$(1,-2,2,0)$&$A'$&$\cO$&$(0,0,0,0)$\\\hline
$B$&$K_2+2K_3$&$(1,0,-2,2)$&$B'$&$K_2+K_3$&$(1,-1,0,1)$\\\hline
$C$&$K_1+K_2+2K_3$&$(-1,1,-2,2)$&$C'$&$K_1+K_2+K_3$&$(-1,0,0,1)$\\\hline
$D$&$K_2+2K_3+2K_4$&$(1,0,0,-2)$&$D'$&$K_2+2K_3+K_4$&$(1,0,-1,0)$\\\hline
$E$&$K_1+K_2+2K_3+2K_4$&$(-1,1,0,-2)$&$E'$&$K_1+K_2+2K_3+K_4$&$(-1,1,-1,0)$\\\hline
$F$&$K_1+2K_2+2K_3+2K_4$&$(0,-1,2,-2)$&$F'$&$K_1+2K_2+2K_3+K_4$&$(0,-1,1,0)$\\\hline
$G$&$K_1+2K_2+4K_3+2K_4$&$(0,1,-2,0)$&$G'$&$K_1+2K_2+3K_3+2K_4$&$(0,0,0,-1)$\\\hline
&&&$H'$&$K_2+K_3+K_4$&$(1,-1,1,-1)$\\\hline
&&&$I'$&$K_1+K_2+K_3+K_4$&$(-1,0,1,-1)$\\
\hline
\end{tabular}}\par\medskip

Let us consider a reduced word $\el=(l_1,\dots,l_{24})$ of maximal length  for $\DF_4$, and a subword $\el'=(l_1,\dots,l_{i}=:j)$. Then we can give the following precise description of the descent rules given in Proposition \ref{prop:descent}:

\begin{itemize}
\item[(C)] (change of degree) 
$$
h^1(Z_{\el'},e^L)=h^0(Z_{\el'[1]},e^{L'})\quad\mbox{denoted by}\quad e^L\stackrel{j}{\lra}e^{L'}$$
in the following cases:
$$
\xymatrix{A\ar[d]^2&B\ar[d]^3&C\ar[d]^3&D\ar[d]^4&E\ar[d]^4&F\ar[d]^4&G\ar[d]^3\\A'&B'&C'&D'&E'&F'&G'}
$$
\item[(V1)] (vanishing, degree 1)
$$
h^1(Z_{\el'},e^L)=0\quad\mbox{denoted by}\quad e^L\stackrel{j}{\lra}0$$
in the following cases:
$$
\xymatrix{C\ar[d]^1&E\ar[d]^1&F\ar[d]^2\\0&0&0}
$$
\item[(V0)] (vanishing, degree 0)
$$
h^0(Z_{\el'},e^{L'})=0\quad\mbox{denoted by}\quad e^{L'}\stackrel{j}{\lra}0$$
in the following cases:
$$
\xymatrix{B'\ar[d]^2&C'\ar[d]^1&D'\ar[d]^3&E'\ar[d]^{1\mbox{ or }3}&F'\ar[d]^2&G'\ar[d]^4&H'\ar[d]^{2\mbox{ or }4}&I'\ar[d]^{1\mbox{ or }4}\\0&0&0&0&0&0&0&0}
$$

\item[(D1)] (descent, degree $1$)
$$
h^1(Z_{\el'},e^L)\leq h^1(Z_{\el'[1]},e^L)+h^1(Z_{\el'[1]},e^M)\quad\mbox{denoted by}\quad e^L\stackrel{j}{\lra}e^M$$
in the following cases:
$$
\xymatrix{A\ar[r]^3&B\ar[r]^4\ar[d]^1&D\ar[d]^1&&\\&C\ar[r]^4&E\ar[r]^2&F\ar[r]^3&G}
$$

\item[(D0)] (descent, degree $0$)
$$
h^0(Z_{\el'},e^{L'})\leq h^0(Z_{\el'[1]},e^{L'})+h^0(Z_{\el'[1]},e^{M'})\quad\mbox{denoted by}\quad e^{L'}\stackrel{j}{\lra}e^{M'}$$
in the following cases:
$$
\xymatrix{B'\ar[r]^4\ar[d]^1&H'\ar[r]^3\ar[d]^1&D'\ar[d]^1&&\\C'\ar[r]^4&I'\ar[r]^3&E'\ar[r]^2&F'\ar[r]^3&G'}
$$
\end{itemize}

\subsection{An appropriate word}\label{ssec:theword}

Let us consider the word $$\el_0:=(2,1,2,4,3,4,2,3,4,2,3,1,2,3,2,1,4,3,2,1,3,4,2,3),$$ and its subword
$$
\el:=\el_0[15]=(2,1,2,4,3,4,2,3,4).
$$  
Let us start by describing the result of the descent rules for the cohomology, for all the steps of the construction of the corresponding Bott--Samelson varieties in which the last letter added is equal to $2$. We do not consider the first appearance of the letter $2$, for which the construction of the Bott--Samelson variety is unique. The precise descent procedures for the rest of the appearances of the letter  $2$ are described in Appendix \ref{app:descent}. 
%The tables must be read from right to left; a one in the upper part of the table indicates the presence of the first cohomology group of the corresponding divisor at the corresponding step, while a one in the lower part of the table indicates the presence of the zeroth cohomology group of the corresponding divisor at the corresponding step.

The first two tables there tell us that the dimension of the groups $H^1(Z_{\el_0[22]},e^{K_2})$ and $H^1(Z_{\el_0[18]},e^{K_2})$ is at most $1$ (coming from the global sections of the trivial line bundle $A'$) plus the sum of the $h^1$'s of the divisors $A,B,C,D,E,F,G$ on a point. We conclude then that $Z_{\el_0[21]}$ and $Z_{\el_0[17]}$ are flag--compatible with $\DF_4(D)$.

The third table, however, tells us only that $h^1(Z_{\el},e^{K_2})\leq 4$. This is in fact the only step in which we do not have unicity, as the last three tables (containing the descent for $h^1(Z_{\el_0[k]},e^{K_2})$, for $k=12,10,6,2$) show.

We then conclude the following 

\begin{proposition}\label{prop:main}
Let $\el_0$ be as above, and assume that $\el_0[14]$ is flag--compatible. Then ${\el_0}$ is flag--compatible.
\end{proposition} 

\begin{remark}\label{rem:admisext}
A priori, the only restriction we have on $H^1(Z_{\el},e^{K_2})$ is the following. The variety $Z_{\el}$ contains the fibers $\gamma_2$ of the projections $Z_{\el_0[i-1]}\to Z_{\el_0[i]}$ ($i=24,22,18$), which are mapped to curves on $\DF_4(D)$ in the class $\Gamma_2$, and the cocycle $\theta\in H^1(Z_{\el},e^{K_2})$ defining $Z_{\el_0[14]}\to Z_{\el}$ satisfies that its restriction to any of these curves is different from zero. 
%the projection $Z_{\el_0[17]}\to Z_{\el_0[18]}$, which are mapped to curves on $\DF_4(D)$ in the class $\Gamma_2$. Following \cite[Remark~3.4]{OSWW}, {\color{red} we have already explained this}the fact that the fibers of $Z_{\el_0[14]}\to Z_{\el_0[15]}=Z_{\el}$ passing by points in $Z_{\el_0[17]}$ map into $\DF_4(D)$ to the same curves as the images of the fibers $\gamma_2$ of $Z_{\el_0[17]}\to Z_{\el_0[18]}$, tells us that a cocycle $\theta\in H^1(Z_{\el},e^{K_2})$ defining $Z_{\el_0[14]}\to Z_{\el}$ satisfies that its restriction to any of these curves is different from zero. 
In other words:
\begin{equation}
\theta\in H^+:=\bigcap_{\gamma_2}H^+(\gamma_2), \quad H^+(\gamma_2):=\left\{\theta'\in H^1(Z_{\el},e^{K_2})\,\,|\,\,\, \theta'_{|\gamma_2}\neq 0\right\}.
\end{equation}
Later on, we will denote by 
$$E^+\subset E^+(\gamma_2)\subset \P((H^1(Z_{\el},e^{K_2}))^\vee), $$ 
the sets of classes of elements in $H^+$ and in $H^+(\gamma_2)$, modulo homotheties. A priori we cannot tell that the conditions imposed by all the curves $\gamma_2$ are the same, but we note that,  for every $\gamma_2$, $E^+(\gamma_2)$ is the complement of a hyperplane in  $\P((H^1(Z_{\el},e^{K_2}))^\vee)$, so that it is isomorphic to $\C^k$, $k\leq 3$.
\end{remark}

%%%%%%%%%%%%%%%%%%%%%%%%%%%%%%%%%%%%%%%%%%%%%%

\section{The Schubert variety ${\Sc_{\el}}$ and its contractions}\label{sec:PandB}

 Along this section, in order to compactify notation, whenever we use explicit expressions of words as subindices (on Bott--Samelson or Schubert varieties), we will avoid the use of commas. We will make use of the following auxiliary statement, for which we refer to \cite[Propositions 5.5, 5.6]{OSWW}:

\begin{lemma}\label{lem:forfiber}
 Let $I\subsetneq D$ be a subset, and consider the contraction $\pi_I:\DF_4(D)\to \DF_4(D\setminus I)$  associated with the face generated by the classes $\Gamma_i$, $i\in I$. Let $\el'$ be a reduced word of maximal length for the subdiagram $\cD_I$ of $\cD$ supported on $I$, and let $Z_{\el'}$ be the corresponding Bott--Samelson variety. Then the evaluation map $f:Z_{\el'}\to \Sc_{\el'}\subset \DF_4(D)$ is birational onto its image, which is a complete flag manifold of type $\cD_I$.
\end{lemma}

We will consider the words $\el_0$ and $\el=\el_0[15]$ introduced in Section \ref{ssec:theword}, and the following auxiliary objects:

\begin{notation}
Set  $$P:=\pi_{23}(\Sc_{\el})
\subset \DF_4(1,4), \quad B:=\pi_{234}(\Sc_{\el})
\subset \DF_4(1) 
$$
\end{notation}

The following lemma describes geometrically the varieties $P$ and $B$.

\begin{lemma}\label{lem:BP}
The subvariety $B$ is isomorphic to $\P^2$, and the restriction of the ample generator of $\DF_4(1)$ to $B$ has degree one. Moreover, the fibers of the natural map $Z_{\el}\to B$ are Bott--Samelson varieties of the form $Z_{(2434234)}$ and  
the natural map $\pi:P\to B$ is a $\P^5$-bundle.\end{lemma}

\begin{proof} 
By Lemma \ref{lem:forfiber} $\Sc_{(212)}\subset \DF_4(D)$ is a complete flag of type $\DA_2$; in particular its image $B$ in $\DF_4(1)$ is isomophic to $\P^2$ and the fibers of  its projection onto $B\subset\DF_4(1)$ are $\P^1$'s in the class $\Gamma_2$.

Now the first assertion follows from the equality:
$$
\pi_{234}(\Sc_{\el})
=\pi_{234}(\Sc_{\el[6]})
=\pi_{234}(\Sc_{(212)})
$$
and the fact that the images of the curves of type $\Gamma_1$ in $\DF_4(1)$ have degree one with respect to its ample generator $H_1$. 
To check the second part we add the subword $(4,3,4,2,3,4)$, to get that the fibers of $Z_{\el}\to B$ are of the required form. 

Finally we need to check that the image of every fiber of $Z_{\el}\to B$ into $B$ is a $\P^5$ (fiber of $\DF_4(1,4)\to \DF_4(1)$). For this, it is enough to note that the word $\el'=(2,4,3,4,2,3,4,2,3)$ is a reduced word of maximal length for the fibers of $\DF_4(D)\to \DF_4(1)$, and that  $Z_{\el'}$ and $Z_{\el'[2]}=Z_{(2434234)}$ have the same image into $\DF_4(1,4)$. 
\end{proof}

\subsection{Identifying the bundle $\bm{\cE}$}\label{ssec:step1}

The variety $P$ is the projectivization of the restriction to the projective plane $B$ of the bundle $\cE$ on $\DF_4(1)$, provided by Lemma \ref{lem:rank6bdl}. Since we are not going to use other restrictions of $\cE$, we will abuse notation and write $\cE$ instead  of $\cE_{|B}$. The following Lemma (in which the restriction of $H_1$ to $B$ is denoted by $\cO_B(1)$) describes the bundle $\cE$ completely.

\begin{lemma}\label{lem:decomp}
The subvariety $P$ is isomorphic to $\P(\cE)$, with:
$$
\cE\cong\big(\cO_{B}\otimes W\big)\oplus T_{B}(-1)\oplus\big(\cO_{B}(1)\otimes W^\vee\big),
$$
for a certain vector space $W$ of dimension two.
\end{lemma}

\begin{proof} Let us first consider the image of $Z_{(2124)}=Z_{\el[5]}\subset Z_{\el}$ into $P$. Since $(2,1,2,4)$ is a reduced word of maximal length for the subdiagram of $\DF_4$ supported at the set of nodes $\{1,2,4\}$, Lemma \ref{lem:forfiber} tells us that $Z_{(2124)}$ maps birationally onto its image into $\DF_4(D)$, which is a complete flag of type $\DA_2\times\DA_1$. From this it follows that its image into $\DF_4(1,4)$ is a projective subbundle of $P\to B$ of the form $P_1:=B\times r_1$, where $r_1$ is a projective line, image of $Z_{(2124)}$ into $\DF_4(4)$. But $\cE$ has been defined as the push forward of the divisor $H_4$, so the inclusion $P_1\subset P$ is given by a surjective morphism 
$$\cE\to\cO_B\otimes W, \mbox{ where }r_1=\P(W).$$ 

On the other hand, following Section \ref{ssec:flagsF41}, we may consider $\DF_4(4)$ and $\DF_4(3)$ as parameter spaces of families of $\Q^5$'s and $\P^2$'s in $\DF_4(1)$, respectively. The surface $B=\pi_{234}(\Sc_{(212)})\subset\DF_4(1)$ is, by construction, one of the $\P^2$'s parametrized by $\DF_4(3)$; then, as we have seen in the last item of Section \ref{ssec:flagsF41}, we may assert that the line $r_1\subset \DF_4(4)$ may be described as follows:
$$
r_1=\{\Q^5\in\DF_4(4)|\,\,\Q^5\supset B \}.
$$

Let us denote by $\cF_2^{\,\vee}(1)$ the kernel of $\cE\to\cO_B\otimes W$. 
The skew-symmetric form  $\eta$ provides a commutative diagram, with exact rows and columns:

\begin{equation}
\xymatrix@=30pt{
\cO_B(1)\otimes W^\vee\ar[d]\ar@{=}[r]&\cO_B(1)\otimes W^\vee\ar[d]&\\
\cF_2^{\,\vee}(1)\ar[r]\ar[d]&\cE\cong\cE^{\vee}(1)\ar[r]\ar[d]&\cO_B\otimes W \ar@{=}[d]\\
\cG_2\ar[r]&\cF_2\ar[r]&\cO_B\otimes W \\
}\label{eq:diagram}
\end{equation}
for a certain rank two vector bundle $\cG_2$, satisfying $\cG_2=\cG_2^\vee(1)$; this last condition easily implies that, for any line $\ell$ in $B$,  $\cG_2|_{\ell} \simeq \cO_\ell \oplus \cO_{\ell}(1)$. By the classification of uniform vector bundles on $\P^2$ it follows that either $\cG_2 \simeq T_{\P^2}(-1)$ or $\cG_2 \simeq \cO_{\P^2} \oplus \cO_{\P^2}(1)$. In the second case we would get that $\cF_2$ and $\cE$ are split too; in particular we would have $\cE \simeq \cO_{\P^2}^{\oplus 3} \oplus \cO_{\P^2}^{\oplus 3}(1)$, from which we would get the contradiction: 
$$\dim \{\Q^5\in\DF_4(4)|\,\,\Q^5\supset B \} \geq 2.$$ 
 
Finally, since $H^1(T_{\P^2}(-1))=H^1(\cO_{\P^2})=H^1(\cO_{\P^2}(1))=0$ we get first that $\cF_2 \simeq T_{\P^2}(-1) \oplus  \big(\cO_{\P^2}\otimes W\big)$, and subsequently that 
$\cE\cong\big(\cO_{\P^2}\otimes W\big)\oplus T_{\P^2}(-1)\oplus\big(\cO_{\P^2}(1)\otimes W^\vee\big)$.
\end{proof}

\begin{remark}\label{rem:isoP1P2}

Setting  $\cF_1:=\cO_B\otimes W$, the proof of the Lemma tells us that the vector bundle $\cE$ comes with a filtration
$$
\cF_1^\vee(1)\hookrightarrow \cF_2^\vee(1)\hookrightarrow\cE.
$$
Equivalently, we have surjections $\cE\lra \cF_2\lra \cF_1$,
whose projectivizations provide a flag of projective subbundles over $B$:
 $$
 P_1=B\times r_1 =\P(\cF_1)\hookrightarrow P_2:= \P(\cF_2)\hookrightarrow P.
 $$
With the notation introduced in Definition \ref{def:isotr}, we may then write 
$P_1\subset P_2=P_1^\perp$ and say that $P_1$ is isotropic.
\end{remark} 

\subsection{Isotropic lines in $\bm{P}$}\label{ssec:isotP}

In the sequel we will denote by $\cM\subset\DF_4(1,3)$ the subvariety parametrizing isotropic lines contained in $P$ (see Remark \ref{rem:isolines}). The corresponding universal family $\cU $ is the inverse image of $\cM$ in $\DF_4(1,3,4)$: 
$$
\xymatrix@C=40pt{\cU \,\,\ar@{^{(}->}[r]\ar[d]&\DF_4(1,3,4)  \ar[d]^{\pi_{134,4}} \ar[r]^{\pi_{134,3}}&\DF_4(1,4)\\
\cM  \,\,\ar@{^{(}->}[r]&\DF_4(1,3)
}
$$
Moreover, we will also consider the subfamilies parametrized by:
\begin{equation}
\begin{array}{l}\vspace{0.2cm}
\cM_1:=\{\ell \in\cM\,\,|\,\,\,\ell\cap P_1\neq\emptyset,\,\,\,\ell\subset P_2\},\\
\cM_2:=\{\ell \in\cM\,\,|\,\,\,\ell\cap P_2\neq\emptyset\}.
\end{array} 
\label{eq:M1M2}
\end{equation}
The next two lemmata allow us to describe these subfamilies in terms of Schubert varieties.

\begin{lemma}\label{lem:M1} The subvariety $\cM_1 \subset \DF_4(1,3)$ is the image of $\Sc_{\el[4]}=\Sc_{(21243)}$ via the contraction $\pi_{24}:\DF_4(D)\to\DF_4(1,3)$. The universal family $\cU_1$ over $\cM_1$ is the image of $\Sc_{\el[3]}=\Sc_{(212434)}$ via the contraction $\pi_{2}:\DF_4(D)\to\DF_4(1,3,4)$.
\end{lemma}

\begin{proof}
As in Lemma \ref{lem:BP} we can show that $\pi_{234}(\Sc_{\el[4]}) = B$ and that the fibers of the natural map $Z_{\el[4]} \to B$ are Bott--Samelson varieties of the form  $Z_{(243)}$.
Let us describe the image $\Sc_{(243)}$ of a variety of the form $Z_{(243)}$ in $\DF_4(D)$, by retracing the way in which it is constructed. We will use the description of vertical isotropic flags, and of the fibers of the maps $\pi_2,\pi_3,\pi_4$, that we have presented in Section \ref{ssec:vertiso}. 

Let us start with an element $b \in B$ and a point in $\pi_{234}^{-1}(b)$, which may be described as a flag of isotropic subspaces in the fiber $P_b$ of $P$ over $B$:
$$\Sc_{\emptyset}=\{p_0 \subset \ell_0 \subset \pi_0 \subset P_b\}.$$
The variety $\Sc_{(2)}$ can then be constructed upon the pencil of isotropic planes containing $\ell_0$ and contained in $\ell_0^\perp$,
$$\Sc_{(2)}=\{p_0 \subset \ell_0 \subset \pi \,\,|\,\,   \pi \subset  \ell_0^\perp\},$$
and $\Sc_{(24)}$ is constructed by moving the point on the line $\ell_0$,
$$\Sc_{(24)}=\{p \subset \ell_0 \subset \pi \,\,|\,\,    \pi\subset  \ell_0^\perp\}.$$
Finally,  we have the following description of $\Sc_{(243)}$: 
$$\Sc_{(243)}=\{p \subset \ell \subset \pi \,\,|\,\, p\in\ell_0, \,\,\pi\subset \ell_0^\perp \}.$$
In particular, $\pi_{24}(\Sc_{(243)})$ parametrizes the isotropic lines meeting a given line $\ell_0$ and contained in the orthogonal three-dimensional linear space $\ell_0^\perp$; moreover, the line $\ell_0$ can be described as $\pi_{23}(\Sc_{(24)})$.

Going back to our statement, 
the above argument tells us that in order to conclude, we just need to identify the fiber of $\pi_{23}(\Sc_{(2124)})\to B$, and its orthogonal subspace, at every point $b\in B$. But now, 
from the proof of Lemma \ref{lem:decomp} we know that $B\times r_1=P_1= \pi_{23}(\Sc_{(2124)})$. This tells us that over every point $b\in B$, the fiber of $\pi_{24}(\Sc_{(21243)})\to B$  consists of the isotropic lines meeting the fiber of $\pi_{23}(\Sc_{(2124)})\to B$, which is $P_{1,b}$, and contained in its orthogonal, which is $P_{2,b}$. This concludes the first part of the statement. For the second we first note that, arguing as above, the fiber of $\Sc_{(212434)}\to B$ can be written as:
$$\Sc_{(2434)}=\{p' \subset \ell \subset \pi \,\,|\,\, \ell\in\cM_1,\,\,P_{1,b}\subset\pi\subset P_{2,b} \},$$
and then its image via $\pi_{2}$ is precisely the universal family $\cU_1$. 
\end{proof}

\begin{lemma}\label{lem:famM2}  
The subvariety $\cM_2 \subset \DF_4(1,3)$ is the image of $Z_{\el[1]}$ via $\pi_{24}$, and its universal family $\cU_2$ over is $\pi_{2}(\Sc_{\el})\subset\DF_4(1,3,4)$. \end{lemma}

\begin{proof}
As in the previous lemma, the second part follows from the first. 
Consider the family of isotropic lines in $P$: 
$$
\xymatrix@C=15mm{\cM&\cU\ar[r]^{\pi_{134,3}}\ar[l]_{\pi_{134,4}}&P}
$$
For simplicity, let us set here $p:=\pi_{134,4}$, and $q:=\pi_{134,3}$. The variety $\cM_2$ and its universal family $\cU_2$ can be seen as subvarieties of $\cM\subset\DF_4(1,3)$ and of $\cU\subset\DF_4(1,3,4)$, respectively, in the following way: 
$$\cM_2=p(q^{-1}(P_2)), \quad \cU_2=p^{-1}(\cM_2).$$
From the previous Lemma, 
we have $q(\pi_2(\Sc_{\el[1]}))=\pi_{23}(\Sc_{\el[1]})=\pi_{23}(\Sc_{\el[3]})$, which is equal to  $P_2$, hence   
$\pi_2(\Sc_{\el[1]}
) \subseteq q^{-1}(P_2))$.
In order to conclude the proof, it suffices to show that equality holds, for which it is enough to check that $\dim \pi_2(\Sc_{\el[1]})=\dim q^{-1}(P_2))=8$. Now, since $\el[1]$ is a reduced word, $ \dim \Sc_{\el[1]}=\dim Z_{\el[1]}  = 8$; if $\dim \pi_2(\Sc_{\el[1]})$ were smaller than $\dim \Sc_{\el[1]}$, then the word $(2,1,2,4,3,4,2,3,2)$, obtained by adding an index $2$ to $\el[1]$, would not be reduced, and one can easily compute that this is not the case (see Remark \ref{rem:length}). 
\end{proof}

\subsection{Automorphisms of $\bm{\cE}$}\label{ssec:automor}
We will study here the automorphisms of the vector bundle $$
\cE\cong\big(\cO_{B}(1)\otimes W^\vee\big)\oplus T_{B}(-1)\oplus\big(\cO_{B}\otimes W\big).
$$  
 
We start by noting that, since $$H^0(B,T_B(-2))=H^0(B,\cO_B(-1))=H^0(B,\Omega_B(1))=0,$$ any automorphism of $\cE$ preserves the filtration $\cF_2^\vee(1)\hookrightarrow \cF_1^\vee(1)\hookrightarrow\cE$, and the quotients $\cE\lra \cF_2\lra \cF_1$, so that the corresponding projectivity $\varphi:P\to P$ satisfies $\varphi(P_i)=P_i$, $i=1,2$. Then we have:

\begin{lemma}\label{lem:automatrix}
Every automorphism of $\cE$ is determined by a block matrix:
\begin{equation}
g=\left(\begin{array}{c|c|c}A&h&C\\\hline 0&\alpha&h'\\\hline 0&0&A'\end{array}\right),
\label{eq:automor}
\end{equation}
where:
$$\begin{array}{l}
A \in \Aut(\cO_{B}(1)\otimes W^\vee),\quad\alpha\in\Aut(T_B(-1))\cong\C^*,\quad A' \in \Aut(\cO_{B}\otimes W),\\
 h\in \Hom(T_B(-1),\cO_{B}(1)\otimes W^\vee),\quad h'\in \Hom(\cO_B\otimes W,T_B(-1)),\\ C\in \Hom(\cO_B\otimes W,\cO_B(1)\otimes W^\vee).
 \end{array} 
$$ 
\end{lemma}

\begin{remark}\label{rem:Euler} The entries of the above block matrix may be understood in terms of tensors. In fact, considering $B$ as the Grothendieck projectivization of 
\begin{equation}
V:= H^0(B,\cO_B(1)),
\label{eq:V}
\end{equation}
 and noting that $W=H^0(r_1,{H_4}_{|r_1})$, the Euler sequence on $B$
$$
\shse{\cO_B(-1)}{\cO_B\otimes V^\vee}{T_B(-1)} 
$$
provides the following natural isomorphisms:
$$
\begin{array}{l}
\Aut(\cO_{B}(1)\otimes W^\vee) \cong\GL(W^\vee),\\
\Hom(T_B(-1),\cO_{B}(1)\otimes W^\vee)\cong 
\bigwedge^2 V\otimes W^\vee\subset V\otimes V\otimes W^\vee,\\
\Aut(\cO_{B}\otimes W) \cong\GL(W),\\
\Aut(T_B(-1))\cong\C^*
\subset\GL(V^\vee) \mbox{ (homotheties)},\\
\Hom(\cO_B\otimes W,T_B(-1))\cong V^\vee\otimes W^\vee,\\
\Hom(\cO_B\otimes W,\cO_B(1)\otimes W^\vee)\cong \Hom(W, W^\vee)\otimes V.
\end{array}
$$
We are denoting here by $\bigwedge^2 V\subset V\otimes V$ the vector subspace of skew-symmetric tensors, that is the kernel of the natural map $V\otimes V\to S^2V$. 
\end{remark}

In a similar way, we have a block matrix expression for the skew--symmetric isomorphism $\eta:\cE\to\cE^\vee(1)$. Note that the natural isomorphism (induced by a skew--symmetric form) $J: V^\vee\lra\bigwedge^2V$ induces 
an isomorphism :
$$
J:T_B(-1)\lra \Omega_B(2),
$$
that we keep denoting by $J$. Then we may write:
\begin{lemma}\label{lem:skewmatrix}
The skew--symmetric form $\eta$ is determined by a block matrix:
\begin{equation}
\eta=\left(\begin{array}{c|c|c}0&0&S\\\hline 0&\beta J&-w^t\\\hline -S^t&w&T\end{array}\right),
\label{eq:eta}
\end{equation}
where:
$$
\begin{array}{l}
S\in \Aut(\cO_{B}\otimes W) ,\quad\beta\in\C^*,\\
 w\in \Hom(T_B(-1),\cO_{B}(1)\otimes W^\vee),\\ T\in \Hom(\cO_B\otimes W,\cO_B(1)\otimes W^\vee) \mbox{ (skew--symmetric)}.
 \end{array}
$$
\end{lemma}
In the above lemma the index $t$ (``transposition'') on a morphism of sheaves denotes the twist with $\cO_B(1)$ of the dual morphism. We say that the morphism $T\in \Hom(\cO_B\otimes W,\cO_B(1)\otimes W^\vee)$ is skew--symmetric if considering $T$ as an element of $\Hom(W, W^\vee)\otimes V$ (see Remark \ref{rem:Euler}), it belongs to $\bigwedge^2W^\vee\otimes V$.

We end this section by using automorphisms of $\cE$ to get a convenient expression of the skew--symmetric form $\eta$. 

\begin{lemma}\label{lem:normalform}
There exists an automorphism $\varphi$ of $\cE$ such that, denoting by $I$ the identity in $\Aut(\cO_{B}\otimes W)$, the block matrix associated to $\varphi^t\circ\eta\circ\varphi$ is the following:
\begin{equation}
{\left(\begin{array}{c|c|c}0&0&I\\\hline 0&J&0\\\hline -I&0&0\end{array}\right)}.
\label{eq:normaleta}
\end{equation}
\end{lemma}

\begin{proof}
Let us consider for $\eta$ the expression provided in (\ref{eq:eta}). 
If $U\in \Hom(\cO_B\otimes W,\cO_B(1)\otimes W^\vee)$ is any morphism satisfying that $U^tS-S^tU=T$ (writing $S,T$ and $U$ as tensors, by means of Remark \ref{rem:Euler}, one can easily show that it exists), we can write the matrix $\eta$ as:
$$
\eta=\left(\begin{array}{c|c|c}I&-(S^t)^{-1}w&U\\\hline 0&\sqrt{\beta}&0\\\hline 0&0&S\end{array}\right)^{\hspace{-0.1cm}t}
{\left(\begin{array}{c|c|c}0&0&I\\\hline 0&J&0\\\hline -I&0&0\end{array}\right)}
\left(\begin{array}{c|c|c}I&-(S^t)^{-1}w&U\\\hline 0&\sqrt{\beta}&0\\\hline 0&0&S\end{array}\right),
$$
where $\sqrt{\beta}$ denotes the homothety of $T_B(-1)$ of ratio $\sqrt{\beta}\in\C^*$.
\end{proof}

%%%%%%%%%%%%%%%%%%%%%%%%%%%%%%%%%%%%%%%%%%%%%%

\section{Proof of the main theorem}\label{sec:proof}

 In Remark \ref{rem:admisext} we have defined a subset $E^+\subset \P((H^1(Z_{\el},e^{K_2}))^\vee)$, which parametrizes the possible Bott--Samelson varieties of $\DF_4(D)$ defined by a $1$-cocycle in $H^1(Z_{\el},e^{K_2})$. The idea of our proof is to show that each of them can be constructed by means of a suitable automorphism ${g}\in\Aut(\cE)$. The Bott--Samelson varieties constructed in this way will be all isomorphic as varieties, although not as $\P^1$-bundles over $Z_{\el}$; we may then conclude that $Z_{\el_0[14]}$ is flag--compatible and, by means of Proposition \ref{prop:main}, we conclude the proof of Theorem \ref{thm:main}.

More concretely, we will consider the following subgroup of automorphisms of $\cE$, that is particularly suitable for our purpose:
\begin{equation}
G=\left\{\left.\alpha\left(\begin{array}{c|c|c}I&0&C\\\hline 0&1&0\\\hline 0&0&I\end{array}\right)\right|\,\,\alpha\in\C^*,\,\,\,C\in \Hom(\cO_B\otimes W,\cO_B(1)\otimes W)\right\}.
\label{eq:G}
\end{equation}

\begin{remark}\label{rem:skewG}
By means of Lemma \ref{lem:normalform}, we will assume, without loss of generality, that the block skew--symmetric form $\eta$ takes the form (\ref{eq:normaleta}). Then we may consider the orbit $G\eta$ of $\eta$ by the action  of $G$ on $\Hom(\cE,\cE^\vee(1))$. Given an element $g\in G$, determined by a pair $(\alpha,C)$, it sends $\eta$ to the skew--symmetric form $g(\eta):=g^t\circ\eta\circ g\in G\eta$ given by the block matrix: 
\begin{equation}
g(\eta)=\alpha^2\left(\begin{array}{c|c|c}0&0&I\\\hline 0&J&0\\\hline -I&0&C^t-C\end{array}\right). 
\label{eq:geta}
\end{equation}
\end{remark}

\subsection{Step I: Constructing a morphism $\bm{\psi:G\to E^+}$.} In this section we will consider the Bott--Samelson variety $Z_{\el_0[14]}$ of the complete flag $\DF_4(D)$, and show how to use the elements of the group $G$ to produce deformations of the $\P^1$-bundle $Z_{\el_0[14]}\to Z_{\el_0[15]}=Z_{\el}$.

Set  $\ol{\cU}:=\pi_{23}^{-1}(P)$ and recall that $\cU=\pi_{2}(\pi_{23}^{-1}(P))$, so that we have Cartesian squares:  
$$
\xymatrix@R=30pt{\ol{\cU}\ar[r]\ar@{_{(}->}+<0ex,-2ex>;[d]&\cU\ar[r]\ar@{_{(}->}+<0ex,-2ex>;[d]&P\ar[r]\ar@{_{(}->}+<0ex,-2ex>;[d]&B\ar@{_{(}->}+<0ex,-2ex>;[d]\\
\DF_4(D)\ar[r]&\DF_4(1,3,4)\ar[r]&\DF_4(1,4)\ar[r]&\DF_4(1)}
$$
The variety $\ol{\cU}$  
can be regarded as a subvariety of the family $\ol{F}$ of vertical flags of points, lines and planes in $P$. 
Denoting by $\mu$ the action of $G$ on $P$, we have an induced action of $G$ on $\ol{F}$, and a commutative diagram:
$$
\xymatrix@R=30pt{{\mu}^*\ol{\cU}\ar@{^{(}->}+<+3.5ex,0ex>;[r]\ar[d]&G\times\ol{F}\ar[r]\ar[d]^{\mu}&G\times P\ar[d]^{\mu}\\
\ol{\cU}\ar@{^{(}->}+<+2ex,0ex>;[r]&\ol{F}\ar[r]&P}
$$

\begin{lemma}\label{lem:gpG}
The image of the morphism $\id\times f_{\el}:G\times Z_{\el}\to G\times \ol{\cU}\subset G\times \ol{F}$ lies in $\mu^*\ol{\cU}$, so that we have a commutative diagram:
$$
\xymatrix@R=30pt{&G\times Z_{\el}\ar[d]^{\id\times f_{\el}}\ar[ld]\\{\mu}^*\ol{\cU}\ar@{^{(}->}+<+3.5ex,0ex>;[r]&G\times\ol{F}}
$$
\end{lemma}

\begin{proof}
It is enough to show that $f_\el(Z_\el)$ lies in the fiber of $\mu^*\ol{\cU}$ over $g$, for every $g\in G$. Note that this fiber is precisely the pullback $g^*\ol{\cU}$ of $\ol{\cU}\to P$ via $g:P\to P$. Fix then an element $g\in G$. 

Let us set $P_0:=\pi_{23}(f(Z_{(212)}))\subset P$, which is a section of $P\to B$ (see Lemma \ref{lem:BP}). For every point $b\in B$, we have a flag in $P$ of projective subspaces of the form $P_{0,b}\subset P_{1,b}\subset P_{2,b}\subset P_b$. By the definition of $G$, the projectivities ${g}:P\to P$ associated with elements $g\in G$, preserve these flags.

In particular, for every $b\in B$, ${g}:\ol{F}\to\ol{F}$ preserves the following sets of flags:
\begin{itemize} 
\item $\{P_{0,b}\subset P_{1,b}\subset \pi\,\,|\,\,\pi\subset P_{2,b}\}$,
\item $\{p\subset P_{1,b}\subset \pi\,\,|\,\,\pi\subset P_{2,b}\}$,
\item $\{p\subset \ell\subset \pi\,\,|\,\,p\in P_{1,b},\,\,\ell\in\cM_1,\,\,P_{1,b}\subset\pi\subset P_{2,b}\}$,
\item $\{p'\subset \ell\subset \pi\,\,|\,\,\ell\in\cM_1,\,\,P_{1,b}\subset\pi\subset P_{2,b}\}$ 
\end{itemize}

Recalling the proof of Lemma \ref{lem:M1}, this already tells us that $f_{\el}(Z_{(212434)})$ is contained in ${g}^*\ol{\cU}$. In order to proof that the whole $\Sc_\el=f_{\el}(Z_{\el})$ is contained in ${g}^*\ol{\cU}$, we start by noting that $\Sc_\el$ is constructed by adding to $f_{\el}(Z_{(212434)})$, recursively, the following sets of flags, isotropic with respect to $\eta$:
\begin{itemize} 
\item $\{p' \subset \ell \subset \pi'\,\,|\,\,\ell\in\cM_1\}$, 
\item $\{p' \subset \ell' \subset \pi'\,\,|\,\,p'\in  P_{2,b},\,\,\ell'\in\cM_2,\,\, \dim(\pi'\cap P_{2,b})\geq 1\}$, 
\item $\{p'' \subset \ell' \subset \pi'\,\,|\,\, \ell'\in\cM_2,\,\,\dim(\pi'\cap P_{2,b})\geq 1\}$, 
\end{itemize}
The proof is concluded by showing that these flags are  isotropic also with respect to $g(\eta)$; that is, that the lines $\ell'\in\cM_2$ and the planes $\pi'$ isotropic with respect to $\eta$ meeting $P_{2,b}$ along a line are isotropic with respect to $g(\eta)$.

Note that -- see formula (\ref{eq:geta}) -- ${g}(\eta)$ satisfies that $g(\eta)(u)$ and $\eta(u)$ are proportional, for every $u\in \big(\cO_{B}(1)\otimes W^\vee\big)\oplus T_{B}(-1)$. This is equivalent to say that ${g}(p'^{\perp})={g}(p')^{\perp}$ for every $p'\in P_{2,b}$, and every $b\in B$, that is, that the subspaces orthogonal to points $p\in P_{2,b}$ are the same with respect to $\eta$ and to $g(\eta)$. This obviously implies that the isotropy conditions imposed by $\eta$ and $g(\eta)$ on lines $\ell'$ passing by a point of $P_{2,b}$, are the same (since a line passing by $p'$ is isotropic if and only if it is contained in $p'^\perp$).

Finally, let $\pi'$ be a plane containing  a line $\ell\subset P_{2,b}$. If $\pi'$ is isotropic with respect to $\eta$, then $\ell$ is isotropic with respect to $\eta$, so that $$\ell \subset \pi\subset\ell^\perp, $$ Since $\ell\subset P_{2,b}$, then its orthogonal $\ell^\perp$ with respect to $\eta$ is also its orthogonal with respect to $g(\eta)$, and we conclude that $\pi'$ is isotropic with respect to $g(\eta)$, as well. This concludes the proof.
\end{proof}

\begin{corollary}\label{cor:exists}
There exists a morphism $\psi:G\to E^+$.
\end{corollary}

\begin{proof}
Let us denote by $\cZ$ the pullback of the $\P^1$-bundle $\mu^*\ol{\cU}\to \mu^*{\cU}$ via the composition $f'$ of $\id\times f_\el:G\times Z_\el\to \mu^*\ol{\cU}$ with the natural map $\mu^*\ol{\cU}\to \mu^*{\cU}$, so that we have a Cartesian square, and a section $\sigma:G\times Z_\el\to \cZ$ of the $\P^1$-bundle $\cZ\to G\times Z_\el$:
$$
\xymatrix@=35pt{\cZ\ar[r]\ar[d]&\mu^*\ol{\cU}\ar[d] \\
G\times Z_\el\ar@/^/+<-1.5ex,2ex>;[u]^{\sigma} \ar[ru]_{\id\times f_\el}\ar[r]_{f'}&\mu^*{\cU}}
$$
The $\P^1$-bundle $\cZ\to G\times Z_\el$ is then determined by a cocycle $\theta\in \Ext^1(\cO_{G\times Z_\el},f'^*K_2)$, that can be considered as a family of extensions of $\cO_{Z_\el}$ by $f^*K_2$ parametrized by $G$. Following \cite{Lg83}, %Introduction, Prop. 3.1, Remark 4.6 b)
this family is defined as the pullback of the universal family of extensions  by a certain morphism
$$\psi': G\to H^1(Z_\el,e^{K_2}).$$  
By construction, this morphism assigns to every $g\in G$ a cocycle defining the $\P^1$-bundle obtained by pulling back $\ol{\cU}\to {\cU}$ via the composition of $f_\el:Z_\el\to\ol{\cU}$ with the natural map $\ol{\cU}\to {\cU}$. 
Note that the restriction of this extension to a fiber $\gamma_2$ of $Z_{\el[i-1]}\to Z_{\el[i]}$ ($i=24,22,18$) is nonzero, hence (see Remark \ref{rem:admisext}) $\psi'(G)\subset H^+$. %Note that the restriction of this extension to a fiber $\gamma_2$ of $Z_{\el[17]}\to Z_{\el[18]}$ is nonzero, hence (see Remark \ref{rem:admisext}) $\psi'(G)\subset H^+$. 
Composing with the quotient by homotheties, we get the required morphism $\psi:G\to E^+$.
\end{proof}

\begin{remark}
By construction, for every $g\in G$, the $\P^1$-bundle $Z'_g$ over $Z_\el$ defined by $\psi(g)$ is isomorphic to $Z_{\el_0[14]}$ as a variety (but not necessarily as a $\P^1$-bundle over $Z_\el$). %In particular, $Z'_g$ is flag--compatible. 
\end{remark}

\subsection{Step II: Surjectivity of the map $\bm{\psi}$.}

Let us consider the action of $G$ on $\Hom(\cE,\cE^\vee(1))$ introduced in Remark \ref{rem:skewG}. The orbit $G\eta\subset \Hom(\cE,\cE^\vee(1))$ of $\eta$ is isomorphic to $\C^*\times \big(V\otimes\bigwedge^2 W\big)$ (where $\bigwedge^2 W$ denotes the vector space of anti-symmetric 2-tensors of $W$), and, identifying $G$ with the group $\C^*\times \Hom(\cO_B\otimes W,\cO_B(1)\otimes W)$, the orbit map $G\to G\eta$ is given by 
$$
(\alpha,C)\mapsto (\alpha^2,C^t-C).
$$

\begin{lemma}\label{lem:psifactors}
The morphism $G\to E^+$ factors via 
$V\otimes\bigwedge^2 W\cong V\otimes\C\cong V$.
\end{lemma}

\begin{proof}
If two elements ${g}_1,{g}_2\in G$ satisfy that ${g}_1(\eta)$ and ${g}_2(\eta)$ are proportional, then they provide the same isotropy condition on flags in $P$, and consequently the same $\P^1$-bundles $g_i^*\ol{\cU}$, $i=1,2$; in other words, we have $\psi({g}_1)=\psi({g}_2)$. This implies that $\psi$ factors via the quotient $G\eta/\C^*$ of the orbit $G\eta$ modulo homotheties, which is isomorphic to $V$.
\end{proof}

The next statement tells us that the induced map from $V$ to $E^+$ is injective:
 
\begin{lemma}\label{lem:psiinj}
Let $g_i\in G$, $i=1,2$ be two automorphisms satisfying that ${g}_1(\eta)$ and ${g}_2(\eta)$ are not proportional.  Then the $\P^1$-bundles $Z'_i$, $i=1,2$ defined by them
$$
\xymatrix{Z'_i\ar[r]\ar[d]_{f'_i}&Z_{\el}\ar[d]^{\pi_2\circ f'_i}\\{g}_i^*\ol{\cU}\ar[r]^{\pi_2}
&{g}_i^*\cU
} 
$$
are not isomorphic as bundles over $Z_\el$.
\end{lemma}

\begin{proof}
If ${g}_1(\eta)$ and ${g}_2(\eta)$ are not proportional, it follows that there exists $b\in B$ such that ${g}_1(\eta)$ and ${g}_2(\eta)$ do not provide the same isotropy condition on the fiber $P_b$. In particular, a general line $r\subset P_b$ isotropic with respect to ${g}_1(\eta)$ is not isotropic with respect to ${g}_2(\eta)$. Note that this line does not meet $P_{2,b}$, since isotropic lines meeting $P_{2,b}$ are the same for both forms.

Let $r'_i$ denote the intersection of $P_{2,b}$ with the subspace orthogonal to $r$ with respect to ${g}_i(\eta)$, for $i=1,2$; in both cases $r'_i$ is a line, since if $r'_i$ had bigger dimension, then it would meet $P_{1,b}$, and consequently $r$ would meet $P_{2,b}$. Note that $r'_i$ is precisely the set of points $p\in P_{2,b}$ satisfying that every line joining $p$ with a point of $r$ is isotropic with respect to ${g}_i(\eta)$, for $i=1,2$. Since isotropic lines meeting $P_{2,b}$ are the same for both forms, it follows that $r'_1=r'_2$; let us denote it simply by $r'$.  

Let us fix a point $q\in r'$, and consider the plane $q+r$, and the family of lines $C$ passing by $q$ in $q+r$. By construction, $C$ is a family of lines isotropic with respect to ${g}_i(\eta)$, $i=1,2$ and there is an injective morphism $C\to{g}_i^*\cU$, sending an element $\ell'\in C$ to $(q,\ell')$ (here we are interpreting ${g}_i^*\cU$ as the universal family of isotropic lines with respect to ${g}_i(\eta)$); abusing notation, we denote by $C\subset {g}_i^*\cU$ its image.  Note that $C$ lies on a fiber of ${g}_i^*\cU\to P$, for each $i$, which is isomorphic to $\P^3$, and that the pullback of ${g}_i^*\ol{\cU}\to{g}_i^*\cU$ to this $\P^3$ is the projectivization of a null correlation bundle (see Section \ref{sec:B2}). 

Moreover, since all the elements of the family $C$ meet $P_{2,b}$, Lemma \ref{lem:famM2} tells us that $C$ belongs to the image of $Z_{\el}$, hence we may now conclude by showing that the pullback of ${g}_i^*\ol{\cU}\to{g}_i^*\cU$ to $C$ is different in the cases $i=1,2$. This is done by noting that the plane $q+r$ is isotropic with respect to ${g}_1(\eta)$, which means that $C\subset {g}_1^*\cU$ is isotropic with respect to the null correlation bundle mentioned above and, in particular, the restriction of the bundle ${g}_1^*\ol{\cU}\to{g}_1^*\cU$ to $C$ is isomorphic to $\P(\cO_C\oplus\cO_C(2))$. On the other hand, $q+r$ is not isotropic with respect to ${g}_2(\eta)$, and so $C$ is a non isotropic line, therefore the restriction of the bundle ${g}_2^*\ol{\cU}\to{g}_2^*\cU$ to $C$ is isomorphic to $\P(\cO_C(1)\oplus\cO_C(1))$.
\end{proof}

We may now achieve the goal of this section:

\begin{corollary}
The map $\psi:G\to E^+$ is surjective.
\end{corollary}

\begin{proof}
Let $\gamma_2\subset Z_\el$ be a  fiber of $Z_{\el[i-i]}\to Z_{\el[i]}$, $i=24,22,18$ (see Remark \ref{rem:admisext}), 
%Let $\gamma_2\subset Z_\el$ be a  fiber of $Z_{\el[17]}\to Z_{\el[18]}$ (see Remark \ref{rem:admisext}), 
and let us consider the composition of  $\psi$  with the inclusion  $E^+\hookrightarrow E^+(\gamma_2)$, which factors via $V\cong\C^3$ by Lemma \ref{lem:psifactors}; this morphism from $V$ to $E^+(\gamma_2)$ is injective by Lemma \ref{lem:psiinj}. Since $E^+(\gamma_2)\cong\C^k$, $k\leq 3$, the Ax--Grothendieck theorem (\cite[Proposition 10.4.11]{Gro2}, \cite{Ax69}) tells us that $k=3$, and that this map is surjective. It follows that $E^+=E^+(\gamma_2)$, and that $\psi$ is surjective, as well. 
\end{proof}

%%%%%%%%%%%%%%%%%%%%%%%%%%%%%%%%%%%%%%%%%%%%%%

\appendix

\section{Descent tables}\label{app:descent}

This Appendix contains the descent tables for the appearances of the letter $2$ in the word 
$$
\el_0=(2,1,2,4,3,4,2,3,4,2,3,1,2,3,2,1,4,3,2,1,3,4,2,3)
$$
The tables must be read from right to left; a one in the upper part of the table indicates the presence of the first cohomology group of the corresponding divisor at the corresponding step, while a one in the lower part of the table indicates the presence of the zeroth cohomology group of the corresponding divisor at the corresponding step.

%\vspace{-0.3cm}

\begin{table}[!ht]
\scalebox{0.69}{
\begin{tabular}{|c|}
\hline
%&{}&1&2&3&4&5&6&7&8&9&1{}&11&12&13&14&15&16&17\\\hline\hline
 \\\hline\hline
A \\\hline\hline
%G \\\hline
A'\\\hline
%B'\\\hline
%C'\\\hline
%D'\\\hline
%E'\\\hline
%F'\\\hline
%G'\\\hline
%H'\\\hline
%I'\\\hline
\end{tabular}%
\begin{tabular}{||c||c|c|c|}
\hline
  & 2&1&2\\\hline\hline
{}& 1& 1&\\\hline\hline
 \bf{1}& {}& {}&\\\hline

\end{tabular}
}
\qquad
\scalebox{0.69}{
\begin{tabular}{|c|}
\hline
%&{}&1&2&3&4&5&6&7&8&9&1{}&11&12&13&14&15&16&17\\\hline\hline
 \\\hline\hline
A \\\hline
B \\\hline
C \\\hline
D \\\hline
E \\\hline
F \\\hline\hline
%G \\\hline
A'\\\hline
%B'\\\hline
%C'\\\hline
%D'\\\hline
%E'\\\hline
%F'\\\hline
%G'\\\hline
%H'\\\hline
%I'\\\hline
\end{tabular}%
\begin{tabular}{||c||c|c|c|c|c|c|c|}
\hline
   &2&1&2&4&3&4&2\\\hline\hline
{}& {}& {}& 1& 1& 1& 1& \\\hline
   1& 1& 1& 1& 1& {}&{}&\\\hline
   1& 1& {}& {}& {}& {}& {}&\\\hline
   1& 1& 1& 1& {}& {}& {}&\\\hline
   1& 1& {}& {}& {}& {}& {}&\\\hline
   1& {}& {}& {}& {}& {}&{}&\\\hline\hline
%   {}& {}& {}& {}& {}& {}& {}&\\\hline\hline
   \bf{1}& 1& 1& {}& {}& {}&{}&\\\hline
%   {}& {}& {}& {}& {}& {}& \\\hline
%   {}& {}& {}& {}& {}& {}& \\\hline
%   {}& {}& {}& {}& {}& {}& \\\hline
%   {}& {}& {}& {}& {}& {}& \\\hline
%   {}& {}& {}& {}& {}& {}& \\\hline
%   {}& {}& {}& {}& {}& {}& \\\hline
%   {}& {}& {}& {}& {}& {}& \\\hline
%   {}& {}& {}& {}& {}& {}&  \\\hline
\end{tabular}
}
\qquad
\scalebox{0.69}{
\begin{tabular}{|c|}
\hline
%&{}&1&2&3&4&5&6&7&8&9&1{}&11&12&13&14&15&16&17\\\hline\hline
 \\\hline\hline
A \\\hline
B \\\hline
%C \\\hline
D \\\hline\hline
%E \\\hline
%F \\\hline
%G \\\hline
A'\\\hline
B'\\\hline
D'\\\hline
E'\\\hline
F'\\\hline
%G'\\\hline
H'\\\hline
%I'\\\hline
\end{tabular}%
\begin{tabular}{||c||c|c|c|c|c|c|c|c|c|c|}
\hline
   &2&1&2&4&3&4&2&3&4&2\\\hline\hline
   {}& {}& {}& {}& {}& {}& {}& 1& 1& 1& \\\hline
   {}& {}& {}& {}& {}& {}& 1& 1& 1& {}& \\\hline
 %  {}& {}& {}& {}& {}& {}& {}& {}& {}& {}& \\\hline
   {}& {}& {}& {}& {}& 1& 1& {}& {}& {}& \\\hline\hline
 %  {}& {}& {}& {}& {}& {}& {}& {}& {}& {}& \\\hline
 %  {}& {}& {}& {}& {}& {}& {}& {}& {}& {}& \\\hline
 %  {}& {}& {}& {}& {}& {}& {}& {}& {}& {}& \\\hline
    \bf{1}& 1& 1& 1& 1& 1& 1& {}& {}& {}& \\\hline
   {}& {}& {}& 1& 1& {}& {}& {}& {}& {}& \\\hline
    \bf{1}& 1& 1& 1& {}& {}& {}& {}& {}& {}& \\\hline
    \bf{1}& 1& {}& {}& {}& {}& {}& {}& {}& {}& \\\hline
    \bf{1}& {}& {}& {}& {}& {}& {}& {}& {}& {}& \\\hline
 %  {}& {}& {}& {}& {}& {}& {}& {}& {}& {}& \\\hline
   {}& {}& {}& 1& {}& {}& {}& {}& {}& {}& \\\hline
%   {}& 1& 1& {}& {}& {}& {}& {}& {}& {}&  \\\hline
   \end{tabular}
}

\caption{Descent for the groups $H^1(Z_{\el_0[k]},e^{K_2})$, $k=22,18,15$.}
\end{table}

%\vspace{-0.3cm}

\begin{table}[!ht]
\scalebox{0.69}{
\begin{tabular}{|c|}
\hline
%&{}&1&2&3&4&5&6&7&8&9&1{}&11&12&13&14&15&16&17\\\hline\hline
 \\\hline\hline
A \\\hline
B \\\hline
%C \\\hline
D \\\hline\hline
%E \\\hline
%F \\\hline
%G \\\hline
A'\\\hline
B'\\\hline
%C'\\\hline
D'\\\hline
%E'\\\hline
%F'\\\hline
%G'\\\hline
%H'\\\hline
%I'\\\hline
\end{tabular}%
\begin{tabular}{||c||c|c|c|c|c|c|c|c|c|c|c|c|c|c|c|}
\hline
 &2&1&2&4&3&4&2&3&4&2&3&1&2\\\hline\hline
 {}& {}& {}& {}& {}& {}& {}& {}& {}& {}& 1& 1& 1&\\\hline
   {}& {}& {}& {}& {}& {}& {}& {}& 1& 1& 1& {}& {}&\\\hline
   {}& {}& {}& {}& {}& {}& 1& 1& 1& {}& {}& {}& {}&\\\hline\hline
    \bf{1}& 1& 1& 1& 1& 1& 1& 1& 1& 1& {}& {}& {}&\\\hline
   {}& {}& {}& {}& {}& {}& {}& 1& {}& {}& {}& {}& {}&\\\hline
   {}& {}& {}& {}& {}& 1& {}& {}& {}& {}& {}& {}& {}&\\\hline
 
\end{tabular}
}
\scalebox{0.69}{
\begin{tabular}{|c|}
\hline
%&{}&1&2&3&4&5&6&7&8&9&1{}&11&12&13&14&15&16&17\\\hline\hline
 \\\hline\hline
A \\\hline
B \\\hline
C \\\hline\hline
%D \\\hline
%E \\\hline
%F \\\hline
%G \\\hline
A'\\\hline
B'\\\hline
C'\\\hline
%D'\\\hline
E'\\\hline
F'\\\hline
G'\\\hline
%H'\\\hline
I'\\\hline
\end{tabular}%
\begin{tabular}{||c||c|c|c|c|c|c|c|c|c|c|c|c|c|c|c|}
\hline
   &2&1&2&4&3&4&2&3&4&2&3&1&2&3&2\\\hline\hline
  {}& {}& {}& {}& {}& {}& {}& {}& {}& {}& {}& {}& {}& 1& 1&\\\hline
   {}& {}& {}& {}& {}& {}& {}& {}& {}& {}& {}& 1& 1& 1& {}&\\\hline
   {}& {}& {}& {}& {}& {}& {}& {}& {}& {}& {}& 1& {}& {}& {}&\\\hline\hline
   \bf{1}& 1& 1& 1& 1& 1& 1& 1& 1& 1& 1& 1& 1& {}& {}&\\\hline
   {}& {}& {}& {}& {}& {}& {}& {}& {}& {}& 1& {}& {}& {}& {}&\\\hline
   {}& {}& 1& 1& 1& 1& 1& 1& 1& 1& 1& {}& {}& {}& {}&\\\hline
   {}& {}& 1& 1& 1& 1& 1& 1& {}& {}& {}& {}& {}& {}& {}&\\\hline
   {}& 1& 1& 1& 1& 1& 1& {}& {}& {}& {}& {}& {}& {}& {}&\\\hline
   {}& {}& {}& {}& 1& {}& {}& {}& {}& {}& {}& {}& {}& {}& {}&\\\hline
   {}& {}& 1& 1& 1& 1& 1& 1& 1& {}& {}& {}& {}& {}& {}&\\\hline
\end{tabular}
}
\caption{Descent for the groups $H^1(Z_{\el_0[k]},e^{K_2})$, $k=12,10$.}
\end{table}

%\vspace{-0.3cm}

\begin{table}[!ht]
\scalebox{0.69}{
\begin{tabular}{|c|}
\hline
%&{}&1&2&3&4&5&6&7&8&9&1{}&11&12&13&14&15&16&17\\\hline\hline
 \\\hline\hline
A \\\hline
B \\\hline
C \\\hline
D \\\hline
E \\\hline
F \\\hline
G \\\hline\hline
A'\\\hline
B'\\\hline
C'\\\hline
D'\\\hline
E'\\\hline
F'\\\hline
G'\\\hline
%H'\\\hline
%I'\\\hline
\end{tabular}%
\begin{tabular}{||c||c|c|c|c|c|c|c|c|c|c|c|c|c|c|c|c|c|c|c|}
\hline
   &2&1&2&4&3&4&2&3&4&2&3&1&2&3&2&1&4&3&2\\\hline\hline
  {}& {}& {}& {}& {}& {}& {}& {}& {}& {}& {}& {}& {}& {}& {}& 1& 1& 1& 1&\\\hline
   {}& {}& {}& {}& {}& {}& {}& {}& {}& {}& {}& {}& {}& {}& 1& 1& 1& 1& {}&\\\hline
   {}& {}& {}& {}& {}& {}& {}& {}& {}& {}& {}& {}& {}& {}& 1& 1& {}& {}& {}&\\\hline
   {}& {}& {}& {}& {}& {}& {}& {}& {}& 1& 1& 1& 1& 1& 1& 1& 1& {}& {}&\\\hline
   {}& {}& {}& {}& {}& {}& {}& {}& {}& 1& 1& 1& 1& 1& 1& 1& {}& {}& {}&\\\hline
   {}& {}& {}& {}& {}& {}& {}& {}& {}& 1& 1& 1& 1& 1& 1& {}& {}& {}& {}&\\\hline
   {}& {}& {}& {}& {}& {}& {}& {}& 1& 1& 1& 1& 1& 1& {}& {}& {}& {}& {}&\\\hline\hline
   \bf{1}& 1& 1& 1& 1& 1& 1& 1& 1& 1& 1& 1& 1& 1& 1& {}& {}& {}& {}&\\\hline
   {}& {}& {}& {}& {}& {}& {}& {}& {}& {}& {}& {}& {}& 1& {}& {}& {}& {}& {}&\\\hline
   {}& {}& {}& {}& {}& {}& {}& {}& {}& {}& {}& {}& 1& 1& {}& {}& {}& {}& {}&\\\hline
   {}& {}& {}& {}& {}& {}& {}& {}& 1& {}& {}& {}& {}& {}& {}& {}& {}& {}& {}&\\\hline
   {}& {}& {}& {}& {}& {}& {}& {}& 1& {}& {}& {}& {}& {}& {}& {}& {}& {}& {}&\\\hline
   {}& {}& {}& {}& {}& {}& {}& 1& 1& {}& {}& {}& {}& {}& {}& {}& {}& {}& {}&\\\hline
   {}& {}& {}& {}& {}& {}& 2& 2& {}& 1& 1& {}& {}& {}& {}& {}& {}& {}& {}&\\\hline
\end{tabular}
}
\caption{Descent for the group $H^1(Z_{\el_0[6]},e^{K_2})$.}
\end{table}

\begin{table}[!ht]
\scalebox{0.69}{
\begin{tabular}{|c|}
\hline
%&{}&1&2&3&4&5&6&7&8&9&1{}&11&12&13&14&15&16&17\\\hline\hline
 \\\hline\hline
A \\\hline
B \\\hline
C \\\hline\hline
%D \\\hline
%E \\\hline
%F \\\hline
%G \\\hline
A'\\\hline
B'\\\hline
C'\\\hline
%D'\\\hline
E'\\\hline
F'\\\hline
G'\\\hline
H'\\\hline
I'\\\hline
\end{tabular}%
\begin{tabular}{||c||c|c|c|c|c|c|c|c|c|c|c|c|c|c|c|c|c|c|c|c|c|c|c|}
\hline
   &2&1&2&4&3&4&2&3&4&2&3&1&2&3&2&1&4&3&2&1&3&4&2\\\hline\hline
{}& {}& {}& {}& {}& {}& {}& {}& {}& {}& {}& {}& {}& {}& {}& {}& {}& {}& {}& 1& 1& 1& 1&\\\hline
   {}& {}& {}& {}& {}& {}& {}& {}& {}& {}& {}& {}& {}& {}& {}& {}& {}& {}& 1& 1& 1& {}& {}&\\\hline
   {}& {}& {}& {}& {}& {}& {}& {}& {}& {}& {}& {}& {}& {}& {}& {}& {}& {}& 1& 1& {}& {}& {}&\\\hline
  \hline
   \bf{1}& 1& 1& 1& 1& 1& 1& 1& 1& 1& 1& 1& 1& 1& 1& 1& 1& 1& 1& {}& {}& {}& {}&\\\hline
   {}& {}& {}& {}& {}& {}& {}& {}& {}& {}& {}& {}& {}& {}& {}& 1& 1& 1& {}& {}& {}& {}& {}&\\\hline
   {}& {}& {}& {}& {}& {}& {}& {}& {}& {}& {}& {}& 1& 1& 1& 1& 1& 1& {}& {}& {}& {}& {}&\\\hline
   {}& {}& {}& {}& {}& {}& {}& {}& {}& {}& {}& {}& 1& 1& {}& {}& {}& {}& {}& {}& {}& {}& {}&\\\hline
   {}& {}& {}& {}& {}& {}& {}& {}& {}& {}& 1& 1& 1& {}& {}& {}& {}& {}& {}& {}& {}& {}& {}&\\\hline
   {}& {}& {}& {}& {}& {}& {}& {}& {}& 1& 1& {}& {}& {}& {}& {}& {}& {}& {}& {}& {}& {}& {}&\\\hline
   {}& {}& {}& {}& {}& {}& {}& {}& {}& {}& {}& {}& {}& {}& {}& 1& 1& {}& {}& {}& {}& {}& {}&\\\hline
   {}& {}& {}& {}& {}& {}& {}& {}& {}& {}& {}& {}& 1& 1& 1& 1& 1& {}& {}& {}& {}& {}& {}&\\\hline
\end{tabular}
}
\caption{Descent for the group $H^1(Z_{\el_0[2]},e^{K_2})$.}
\end{table}

%%%%%%%%%%%%%%%%%%%%%%%%%%%%%%%%%%%%%%%%%%%%%%%
%\nocite{*} %comment if we don't want the uncited references to appear
\bibliographystyle{plain}
\bibliography{biblio}

\end{document}